\newcommand{\tun}{\begin{picture}(5,0)(-2,-1)
\put(0,0){\circle*{2}}
\end{picture}}

\newcommand{\tdeux}{\begin{picture}(7,7)(0,-1)
\put(3,0){\circle*{2}}
\put(3,5){\circle*{2}}
\put(3,0){\line(0,1){5}}
\end{picture}}

\newcommand{\ttroisun}{\begin{picture}(15,12)(-5,-1)
\put(3,0){\circle*{2}}
\put(6,7){\circle*{2}}
\put(0,7){\circle*{2}}
\put(-0.65,0){$\vee$}
\end{picture}}

\newcommand{\ttroisdeux}{\begin{picture}(5,15)(-2,-1)
\put(0,0){\circle*{2}}
\put(0,5){\circle*{2}}
\put(0,10){\circle*{2}}
\put(0,0){\line(0,1){5}}
\put(0,5){\line(0,1){5}}
\end{picture}}

\newcommand{\tquatreun}{\begin{picture}(15,12)(-5,-1)
\put(3,0){\circle*{2}}
\put(6,7){\circle*{2}}
\put(0,7){\circle*{2}}
\put(3,7){\circle*{2}}
\put(-0.65,0){$\vee$}
\put(3,0){\line(0,1){7}}
\end{picture}}

\newcommand{\tquatredeux}{\begin{picture}(15,18)(-5,-1)
\put(3,0){\circle*{2}}
\put(6,7){\circle*{2}}
\put(0,7){\circle*{2}}
\put(0,14){\circle*{2}}
\put(-0.65,0){$\vee$}
\put(0,7){\line(0,1){7}}
\end{picture}}

\newcommand{\tquatrequatre}{\begin{picture}(15,18)(-5,-1)
\put(3,5){\circle*{2}}
\put(6,12){\circle*{2}}
\put(0,12){\circle*{2}}
\put(3,0){\circle*{2}}
\put(-0.65,5){$\vee$}
\put(3,0){\line(0,1){5}}
\end{picture}}

\newcommand{\tquatrecinq}{\begin{picture}(9,19)(-2,-1)
\put(0,0){\circle*{2}}
\put(0,5){\circle*{2}}
\put(0,10){\circle*{2}}
\put(0,15){\circle*{2}}
\put(0,0){\line(0,1){5}}
\put(0,5){\line(0,1){5}}
\put(0,10){\line(0,1){5}}
\end{picture}}

\newcommand{\tdun}[1]{\begin{picture}(12,5)(-2,-1)
\put(0,0){\circle*{2}}
\put(3,-2){\tiny #1}
\end{picture}}

\newcommand{\tddeux}[2]{\begin{picture}(12,10)(0,-1)
\put(3,0){\circle*{2}}
\put(3,5){\circle*{2}}
\put(3,0){\line(0,1){5}}
\put(6,-2){\tiny #1}
\put(6,4){\tiny #2}
\end{picture}}

\newcommand{\tdtroisun}[3]{\begin{picture}(22,12)(-6,-1)
\put(3,0){\circle*{2}}
\put(6,7){\circle*{2}}
\put(0,7){\circle*{2}}
\put(-0.65,0){$\vee$}
\put(5,-2){\tiny #1}
\put(9,5){\tiny #2}
\put(-7,5){\tiny #3}
\end{picture}}

\newcommand{\tdtroisdeux}[3]{\begin{picture}(12,15)(-2,-1)
\put(0,0){\circle*{2}}
\put(0,5){\circle*{2}}
\put(0,10){\circle*{2}}
\put(0,0){\line(0,1){5}}
\put(0,5){\line(0,1){5}}
\put(3,-2){\tiny #1}
\put(3,4){\tiny #2}
\put(3,11){\tiny #3}
\end{picture}}

\newcommand{\tdquatreun}[4]{\begin{picture}(20,12)(-5,-1)
\put(3,0){\circle*{2}}
\put(6,7){\circle*{2}}
\put(0,7){\circle*{2}}
\put(3,7){\circle*{2}}
\put(-0.6,0){$\vee$}
\put(3,0){\line(0,1){7}}
\put(5,-2){\tiny #1}
\put(8.5,5){\tiny #2}
\put(1,10){\tiny #3}
\put(-5,5){\tiny #4}
\end{picture}}

\newcommand{\tdquatredeux}[4]{\begin{picture}(20,20)(-5,-1)
\put(3,0){\circle*{2}}
\put(6,7){\circle*{2}}
\put(0,7){\circle*{2}}
\put(0,14){\circle*{2}}
\put(-.65,0){$\vee$}
\put(0,7){\line(0,1){7}}
\put(5,-2){\tiny #1}
\put(9,5){\tiny #2}
\put(-5,5){\tiny #3}
\put(-5,12){\tiny #4}
\end{picture}}

\newcommand{\tdquatrequatre}[4]{\begin{picture}(20,14)(-5,-1)
\put(3,5){\circle*{2}}
\put(6,12){\circle*{2}}
\put(0,12){\circle*{2}}
\put(3,0){\circle*{2}}
\put(-.65,5){$\vee$}
\put(3,0){\line(0,1){5}}
\put(6,-3){\tiny #1}
\put(6,4){\tiny #2}
\put(9,12){\tiny #3}
\put(-6,12){\tiny #4}
\end{picture}}

\newcommand{\tdquatrecinq}[4]{\begin{picture}(12,19)(-2,-1)
\put(0,0){\circle*{2}}
\put(0,5){\circle*{2}}
\put(0,10){\circle*{2}}
\put(0,15){\circle*{2}}
\put(0,0){\line(0,1){5}}
\put(0,5){\line(0,1){5}}
\put(0,10){\line(0,1){5}}
\put(3,-2){\tiny #1}
\put(3,3){\tiny #2}
\put(3,9){\tiny #3}
\put(3,15){\tiny #4}
\end{picture}}

\newcommand{\ptroisun}{\begin{picture}(15,12)(-5,-1)
\put(3,7){\circle*{2}}
\put(-0.65,0){$\wedge$}
\put(6,0){\circle*{2}}
\put(0,0){\circle*{2}}
\end{picture}}

\newcommand{\pquatreun}{\begin{picture}(15,12)(-5,-1)
\put(3,7){\circle*{2}}
\put(-0.65,0){$\wedge$}
\put(6,0){\circle*{2}}
\put(0,0){\circle*{2}}
\put(3,0){\circle*{2}}
\put(2.9,0){\line(0,1){7}}
\end{picture}}
\newcommand{\pquatredeux}{\begin{picture}(15,18)(-5,-1)
\put(3,14){\circle*{2}}
\put(-0.65,7){$\wedge$}
\put(6,7){\circle*{2}}
\put(0,7){\circle*{2}}
\put(0,0){\circle*{2}}
\put(0,0){\line(0,1){7}}
\end{picture}}

\newcommand{\pquatrequatre}{\begin{picture}(15,18)(-5,-1)
\put(3,7){\circle*{2}}
\put(-0.65,0){$\wedge$}
\put(6,0){\circle*{2}}
\put(0,0){\circle*{2}}
\put(3,12){\circle*{2}}
\put(3,7){\line(0,1){5}}
\end{picture}}
\newcommand{\pquatrecinq}{\begin{picture}(15,12)(-5,-1)
\put(0,0){\circle*{2}}
\put(7,0){\circle*{2}}
\put(0,7){\circle*{2}}
\put(7,7){\circle*{2}}
\put(0,0){\line(0,1){7}}
\put(7,0){\line(0,1){7}}
\put(.5,1.5){$\scriptstyle \diagup$}
\end{picture}}
\newcommand{\pquatresix}{\begin{picture}(15,12)(-5,-1)
\put(0,0){\circle*{2}}
\put(7,0){\circle*{2}}
\put(0,7){\circle*{2}}
\put(7,7){\circle*{2}}
\put(0,0){\line(0,1){7}}
\put(7,0){\line(0,1){7}}
\put(0,1.5){$\scriptstyle \diagdown$}
\end{picture}}
\newcommand{\pquatresept}{\begin{picture}(15,12)(-5,-1)
\put(0,0){\circle*{2}}
\put(7,0){\circle*{2}}
\put(0,7){\circle*{2}}
\put(7,7){\circle*{2}}
\put(0,0){\line(0,1){7}}
\put(7,0){\line(0,1){7}}
\put(.5,1.5){$\scriptstyle \diagup$}
\put(0,1.5){$\scriptstyle \diagdown$}
\end{picture}}
\newcommand{\pquatrehuit}{\begin{picture}(15,18)(-5,-1)
\put(3,0){\circle*{2}}
\put(-0.65,0){$\vee$}
\put(6,7){\circle*{2}}
\put(0,7){\circle*{2}}
\put(3,14){\circle*{2}}
\put(-0.65,7){$\wedge$}
\end{picture}}
\newcommand{\psix}{\begin{picture}(15,12)(-5,-1)
\put(0,0){\circle*{2}}
\put(7,0){\circle*{2}}
\put(0,7){\circle*{2}}
\put(7,7){\circle*{2}}
\put(14,0){\circle*{2}}
\put(14,7){\circle*{2}}
\put(0,0){\line(0,1){7}}
\put(7,0){\line(0,1){7}}
\put(14,0){\line(0,1){7}}
\put(0.5,1.5){$\scriptstyle \diagup$}
\put(7.5,1.5){$\scriptstyle \diagup$}
\scalebox{1.8}{\put(-0.4,-0.3){$\smallsetminus$}}
\end{picture}}

\newcommand{\pdtroisun}[3]{\begin{picture}(23,12)(-7,-1)
\put(3,7){\circle*{2}}
\put(-0.65,0){$\wedge$}
\put(6,0){\circle*{2}}
\put(0,0){\circle*{2}}
\put(5,5){\tiny #1}
\put(-7,-2){\tiny #2}
\put(9,-2){\tiny #3}
\end{picture}}

\newcommand{\pdquatreun}[4]{\begin{picture}(15,12)(-5,-1)
\put(3,7){\circle*{2}}
\put(-0.65,0){$\wedge$}
\put(6,0){\circle*{2}}
\put(0,0){\circle*{2}}
\put(3,0){\circle*{2}}
\put(2.9,0){\line(0,1){7}}
\put(-6,-2){\tiny #1}
\put(6,6){\tiny #2}
\put(8,-2){\tiny #3}
\put(2,-8){\tiny #4}
\end{picture}}
\newcommand{\pdquatredeux}[4]{\begin{picture}(17,18)(-5,-1)
\put(3,14){\circle*{2}}
\put(-0.65,7){$\wedge$}
\put(6,7){\circle*{2}}
\put(0,7){\circle*{2}}
\put(0,0){\circle*{2}}
\put(0,0){\line(0,1){7}}
\put(-6,-2){\tiny #1}
\put(-6,6){\tiny #2}
\put(8,6){\tiny #3}
\put(5,14){\tiny #4}
\end{picture}}

\newcommand{\pdquatrequatre}[4]{\begin{picture}(15,18)(-5,-1)
\put(3,7){\circle*{2}}
\put(-0.65,0){$\wedge$}
\put(6,0){\circle*{2}}
\put(0,0){\circle*{2}}
\put(3,12){\circle*{2}}
\put(3,7){\line(0,1){5}}
\put(-6,-2){\tiny #1}
\put(8,-2){\tiny #2}
\put(6,6){\tiny #3}
\put(6,13){\tiny #4}
\end{picture}}

\newcommand{\pdquatresix}[4]{\begin{picture}(20,12)(-5,-1)
\put(0,0){\circle*{2}}
\put(7,0){\circle*{2}}
\put(0,7){\circle*{2}}
\put(7,7){\circle*{2}}
\put(0,0){\line(0,1){7}}
\put(7,0){\line(0,1){7}}
\put(.5,1.5){$\scriptstyle \diagup$}
\put(-6,-2){\tiny #1}
\put(9,-2){\tiny #2}
\put(-6,5){\tiny #3}
\put(9,5){\tiny #4}
\end{picture}}

\newcommand{\pdquatresept}[4]{\begin{picture}(20,12)(-5,-1)
\put(0,0){\circle*{2}}
\put(7,0){\circle*{2}}
\put(0,7){\circle*{2}}
\put(7,7){\circle*{2}}
\put(0,0){\line(0,1){7}}
\put(7,0){\line(0,1){7}}
\put(.5,1.5){$\scriptstyle \diagup$}
\put(0,1.5){$\scriptstyle \diagdown$}
\put(-6,-2){\tiny #1}
\put(9,-2){\tiny #2}
\put(-6,5){\tiny #3}
\put(9,5){\tiny #4}
\end{picture}}

\newcommand{\pdquatrehuit}[4]{\begin{picture}(15,18)(-5,-1)
\put(3,0){\circle*{2}}
\put(-0.65,0){$\vee$}
\put(6,7){\circle*{2}}
\put(0,7){\circle*{2}}
\put(3,14){\circle*{2}}
\put(-0.65,7){$\wedge$}
\put(5,-3){\tiny #1}
\put(5,13){\tiny #2}
\put(-6,5){\tiny #3}
\put(8,5){\tiny #4}
\end{picture}}

\font \sevenrm=cmr7
\font \fiverm=cmr5

\documentclass[11pt, english]{amsart}
\input xy
\xyoption{all}
\usepackage{amsfonts}
\usepackage{amssymb}
\usepackage{amsmath}
\usepackage{color}
\usepackage{graphicx}
\usepackage{xspace}
\usepackage{axodraw}
\usepackage{wasysym}
\usepackage{bm}

 \newcommand{\nc}{\newcommand}
 \newcommand{\isoclasse}[1]{\lfloor #1 \rfloor}
   
 {\everymath{\displaystyle\everymath{}}\array}%
 {\endarray}

\newtheorem{thm}{Theorem}

\newtheorem{cor}[thm]{Corollary}

\newtheorem{lem}[thm]{Lemma}
\newtheorem{prop}[thm]{Proposition}

\newtheorem{rmk}[thm]{Remark}

\nc{\comment}[1]{[[{\tt #1}]] }
\nc{\Cal}[1]{{\mathcal {#1}}}
\nc{\mop}[1]{\mathop{\hbox {\rm #1} }\nolimits}
\nc{\gmop}[1]{\mathop{\hbox {\bf #1} }\nolimits}

\nc{\smop}[1]{\mathop{\hbox {\sevenrm #1} }\nolimits}
\nc{\ssmop}[1]{\mathop{\hbox {\fiverm #1} }\nolimits}
\nc{\mopl}[1]{\mathop{\hbox {\rm #1} }\limits}
\nc{\smopl}[1]{\mathop{\hbox {\sevenrm #1} }\limits}
\nc{\ssmopl}[1]{\mathop{\hbox {\fiverm #1} }\limits}
\nc{\frakg}{{\frak g}}
\nc{\g}[1]{{\frak {#1}}}
\def \restr#1{\mathstrut_{\textstyle |}\raise-6pt\hbox{$\scriptstyle #1$}}
\def \srestr#1{\mathstrut_{\scriptstyle |}\hbox to
  -1.5pt{}\raise-4pt\hbox{$\scriptscriptstyle #1$}}
\nc{\wt}{\widetilde} \nc{\wh}{\widehat}

\nc{\redtext}[1]{\textcolor{red}{#1}}
\nc{\bluetext}[1]{\textcolor{blue}{#1}}

\nc\fleche[1]{\mathop{\hbox to #1 mm{\rightarrowfill}}\limits}
\nc{\ignore}[1]{}
\def\semi{\mathrel{\times}\kern -.85pt\joinrel\mathrel{\raise
    1.4pt\hbox{${\scriptscriptstyle |}$}}}
\nc\R{{\mathbb R}}
\nc\N{{\mathbb N}}
\nc\inver{^{-1}}
\nc\point{\hbox{\bf .}}
\nc\un{\hbox{\bf 1}}
\def\link#1#2{\raise -2pt\hbox{$\scriptstyle #1-\!\!-\!\!- #2$}}
\def\slink#1#2{\raise -1.5pt\hbox{$\scriptscriptstyle #1-\!\!\!-\!\!\!- #2$}}

\begin{document}

\title[poset operads]{Operads of finite posets}

\author{Fr\'ed\'eric Fauvet}
\address{IRMA,
10 rue du G\'en\'eral Zimmer,
67084 Strasbourg Cedex, France}
	   \email{fauvet@math.unistra.fr}
 	  \urladdr{}
	 
\author{Lo\"\i c Foissy}
\address{Universit\'e du Littoral - C\^ote d'Opale, Calais}
	   \email{Loic.Foissy@lmpa.univ-littoral.fr}
 	  \urladdr{}
         
\author{Dominique Manchon}
\address{Universit\'e Blaise Pascal,
         C.N.R.S.-UMR 6620, 3 place Vasar\'ely, CS 60026,
         63178 Aubi\`ere, France}       
         \email{manchon@math.univ-bpclermont.fr}
         \urladdr{http://math.univ-bpclermont.fr/~manchon/}

\date{April 27th 2016}

\begin{abstract}
We describe four natural operad structures on the vector space generated by isomorphism classes of finite posets. The three last ones are set-theoretical and can be seen as a simplified version of the first, the same way the NAP operad behaves with respect to the pre-Lie operad. Moreover the two first ones are isomorphic.

\bigskip


\smallskip

\end{abstract}

%
%

\maketitle
\tableofcontents
\section{Introduction}
A finite poset is a finite set $E$ endowed with a partial order $\le$. The Hasse diagram of $(E,\le)$ is obtained by representing any element of $E$ by a vertex, and by drawing a directed edge from $e$ to $e'$ if and only if $e'$ sits directly above $e$, i.e. $e<e'$ and, for any $e''\in E$ such that $e\le e''\le e'$, one has $e''=e$ or $e''=e'$. A poset is \textsl{connected} if its Hasse diagram is connected. Here are the isomorphism classes of connected posets up to four vertices, represented by their Hasse diagrams:
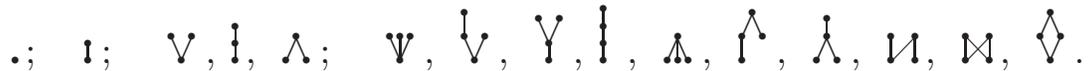
\begin{equation*}
\scalebox{1.3}{\tun; \hskip 3mm \tdeux; \hskip 3mm \ttroisun, \ttroisdeux, \ptroisun; \hskip 3mm
\tquatreun, \tquatredeux, \tquatrequatre, \tquatrecinq, \pquatreun, \pquatredeux, \pquatrequatre, \pquatrecinq, \pquatresept, \pquatrehuit.}
\end{equation*}
Let $\le$ and $\leqslant$ be two partial orders on the same finite set $E$. We say that $\leqslant$ is finer than $\le$ if $x\leqslant y\Rightarrow x\le y$ for any $x,y\in E$. We also say that the poset $(E,\leqslant)$ is finer than the poset $(E,\le)$ and we write $(E,\leqslant)\preceq (E,\le)$. The finest partial order on $E$ is the trivial one, for which any $x\in E$ is only comparable with itself.\\

It is well-known that the linear span $\Cal H$ (over a field $\bm k$) of the isomorphism classes of posets is a commutative incidence Hopf algebra: see \cite[Paragraph 16]{S94}, taking for $\Cal F$ the family of all finite posets with the notations therein. The product is given by the disjoint union, and the coproduct is given by:
\begin{equation}
\Delta_* P=\sum_{P_1\sqcup P_2=P,\, P_1<P_2} P_1\otimes P_2,
\end{equation}
where the sum runs over all \textsl{admissible cuts} of the poset $P$, i.e. partitions of $P$ into two (possibly empty) subposets $P_1$ and $P_2$ such that $P_1<P_2$, which means that for any $x\in P_1$ and $y\in P_2$, we have $x\not\ge y$. For algebraic structures on finite posets, see also \cite{MR11}.\\

The linear species of finite posets $\mathbb P$ is defined as follows: to any finite set $A$ one associates the vector space $\mathbb P_A$ freely generated by all poset structures on $A$. The species structure is obviously defined by relabelling. The natural internal involution is defined as follows: for any finite poset $\Cal A=(A,\le)$ we set $\overline{\Cal A}:=(A,\ge)$.\\

We exhibit in this work several (precisely four) natural operad structures on the linear species $\mathbb P$ of posets. Let us briefly describe them by their partial compositions \cite[Paragraph 5.3.7]{LV11}: the first one, denoted by a circle $\circ$, is obtained by inserting one poset into another one at some vertex summing up over all possibilities. The three others, denoted by a black circle $\bullet$ and two black triangles $\blacktriangledown$ and $\blacktriangle$, are obtained by retaining three different "optimal" insertions, in a way to be precised. The natural involution on posets preserves the set operad $\bullet$, and exchanges $\blacktriangledown$ and $\blacktriangle$.\\

It happens that the operads $(\mathbb P,\circ)$ and $(\mathbb P,\bullet)$ are isomorphic: there is a species automorphism $\Phi$ of $\mathbb P$ such that for any pair $(\Cal A,\Cal B)$ of posets and any $a\in \Cal A$, we have:
\begin{equation}
\Phi(\Cal A\bullet_a\Cal B)=\Phi(\Cal A)\circ_a\Phi(\Cal B).
\end{equation}
The species automorphism $\Phi$ is simply defined by:
\begin{align*}
\Phi_A:\mathbb P_A &\longrightarrow \mathbb P_A\\
\Cal A &\longmapsto \sum_{\Cal A'\preceq\Cal A}\Cal A'.
\end{align*}
We give in Section \ref{sect:comp} some compatibility relations between the three set-theoretic operad structures $\bullet$, $\blacktriangledown$ and $\blacktriangle$. Section \ref{sect:alg} is devoted to the various binary products on free algebras given by the elements of arity two in the operads under investigation. Three associative products and two non-associative permutative (NAP) products are obtained this way. Coproducts obtained by dualization of the aforementioned associative products are also investigated.\\

We give in Section \ref{sect:suboperads} a presentation of the suboperad $\mathbb{WNP}$ of $(\mathbb P, \bullet)$ generated by the elements of arity two. The posets thus obtained are exactly the WN posets, i.e. the posets which do not contain $\pquatresix$ as a subposet \cite{Fposets}. Finally, we give a presentation of the suboperad of $\triangledown$-compatible posets, i.e. the suboperad $\mathbb{CP}_{\triangledown}$ of $(\mathbb P, \blacktriangledown)$ generated by the elements of arity two, as well as a set-theoretical species isomorphism $\theta:\mathbb{WNP}\to\mathbb{CP}_{\triangledown}$ from WN posets onto $\triangledown$-compatible posets.
\section{A first operad structure on finite posets}\label{sect:fp}
\subsection{Reminder on the set operad of sets}
We refer the reader to \cite{M15} for an account of species as well as operads in the species formalism, see also \cite{AM10}. Recall that a set species (resp. a linear species) is a contravariant\footnote{in order to get right actions of permutation groups. It is purely a matter of convention: \cite{AM10} and \cite{M15} prefer the covariant definition.} functor from the category of finite sets together with bijections, into the category of sets (resp. vector spaces) together with maps (resp. linear maps). The collection of finite sets is endowed with a tautological species structure, given by $A\mapsto A$ for any finite set, and $\varphi\mapsto \varphi^{-1}$ for any bijection $\varphi:A\to B$.\\

\noindent Let $A$ and $B$ be two finite sets, and let $a\in A$. We introduce the finite set:
\begin{equation}
A\sqcup_a B:=A\sqcup B\setminus\{a\}.
\end{equation}
It is obvious that the partial compositions $\sqcup_a$ above are functorial, and endow the species of finite sets with an operad structure. Indeed, let $a,a'\in A$ with $a\neq a'$, let $b\in B$, and let $C$ be a third finite set. The parallel associativity axiom reads:
\begin{equation}
(A\sqcup_a B)\sqcup_{a'} C=(A\sqcup_{a'} C)\sqcup_{a} B=A\sqcup B\sqcup C\setminus\{a,a'\}
\end{equation}
and the nested associativity axiom reads:
\begin{equation}
(A\sqcup_a B)\sqcup_{b} C=A\sqcup_{a} (B\sqcup_{b} C)=A\sqcup B\sqcup C\setminus\{a,b\}.
\end{equation}
\subsection{Quotient posets}
Let $P$ be a set and $B\subseteq P$ be a subset. We denote by $P/B$ the set $P\setminus B\sqcup\{B\}$. 
Let $\Cal P=(P,\le)$ be a poset, and let $B$ a nonempty subset of $P$. We define a binary relation $\leqslant$ on $P/B$ as follows: $x\leqslant y$ if and only if:
\begin{itemize}
\item either $x=\{B\}$ and there exists $x'\in B$ such that $x'\le y$,
\item or $y=\{B\}$ and there exists $y'\in B$ such that $x\le y'$,
\item or $x,y\in P\setminus B$ and $x\le y$,
\item or $x,y\in P\setminus B$ and there exist $b,b'\in B$ such that $x\le b$ and $b'\le y$.
\end{itemize}
Recall that the interval $[a,b]$ in a poset $\Cal P=(P,\le)$ (with $a,b\in P$) is defined by:
$$[a,b]:=\{x\in P,\,a\le x\le b\}.$$
We say that $B$ is \textsl{convex} if any interval $[x,y]$ of $\Cal P$ with $x,y\in B$ is included in $B$.
\begin{prop}
Let $\Cal P=(P,\le)$ be a poset, let $B$ be a subset of $P$. Then the binary relation $\leqslant$ defined above is a partial order on $P/B$ if and only if $B$ is convex.
\end{prop}
\begin{proof}
If $B$ is not convex, there exist $x,y,z\in P$ with $x,y\in B$, $z\notin B$ and $x\le z\le y$. By definition of $\leqslant$ we have then $\{B\}\leqslant z\leqslant\{B\}$, hence $\leqslant$ is not antisymmetric.\\

Now suppose that $B$ is convex. Let $x,y,z\in P/B$ with $x\leqslant z$ and $z\leqslant y$. If they are not distinct we immediately get $x\leqslant y$. If $x=\{B\}$, there exists $x'\in B$ with $x'\le z$. Two subcases can occur: If $z\le y$ then $x'\le y$, and then $x=\{B\}\leqslant y$. If there exist $b,b'\in B$ such that $z\le b$ and $b'\le y$, then $x'\le b$ and $b'\le y$, hence $x=\{B\}\leqslant y$ again. The cases $z=\{B\}$ and $y=\{B\}$ are treated similarly. The last case, when the three elements $x,z,y$ are different from $\{B\}$, divides itself into four subcases:
\begin{itemize}
\item if $x\le z$ and $z\le y$, then $x\le y$ hence $x\leqslant y$.
\item if there exist $b,b'\in B$ such that $x\le b$, $b'\le z$ and $z\le y$, then $x\le b$ and $b'\le y$, hence $x\leqslant y$.
\item if there exist $b'',b'''\in B$ such that $x\le z$, $z\le b''$ and $b'''\le y$, then $x\le b''$ and $b'''\le y$, hence $x\leqslant y$.
\item if there exist $b,b',b'',b'''\in B$ such that $x\le b$, $b'\le z$, $z\le b''$ and $b'''\le y$, then $z\in B$ by convexity, which contradicts the hypothesis: this last subcase cannot occur.
\end{itemize}
This proves the transitivity of the binary relation $\leqslant$. Now suppose $x\leqslant y$ and $y\leqslant x$. If $x=\{B\}$, and $y\neq x$, there are $b,b'\in B$ such that $b\le y$ and $y\le b'$. Hence $y\in B$ by convexity, which is a contradiction. The case $y=\{B\}$ is treated similarly. If both $x$ and $y$ are different from $\{B\}$, four subcases can occur:
\begin{itemize}
\item if $x\le y$ and $y\le x$, then $x=y$.
\item If there exist $b,b'\in B$ with $x\le b$, $b'\le y$ and $y\le x$, then $b'\le y\le x\le b$, hence $x,y\in B$ by convexity, which is impossible.
\item If there exist $b'',b'''\in B$ with $y\le b''$, $b'''\le x$ and $x\le y$, then $b'''\le x\le y\le b''$, hence $x,y\in B$ by convexity, which is impossible.
\item If there exist $b,b',b'',b'''\in B$ such that $x\le b$, $b'\le y$ and $y\le b''$, $b'''\le x$, then $y\in B$ by convexity, which is impossible.
\end{itemize}
Hence $\leqslant$ is a partial order.
\end{proof}
For any poset $\Cal P=(P,\le)$ and any convex subset of $P$, we will denote by $\Cal P/B$ the poset $(P/B,\leqslant)$ thus obtained.
\subsection{The first poset operad}\label{sect:poset-operad}
Let $\Cal A=(A,\le_A)$ and $\Cal B=(B,\le_B)$ be two finite posets, and let $a\in A$. We denote by $\Omega(\Cal A,a,\Cal B)$ the set of all partial orders $\le$ on $A\sqcup_a B$ such that:
\begin{enumerate}
\item $\le\restr{B}=\le_B$,
\item $B$ is convex in $(A\sqcup_a B,\le)$,
\item $(A\sqcup_a B,\le)/B\sim \Cal A$, where the set $(A\sqcup_a B)/B=(A\setminus\{a\})\sqcup \{B\}$ is naturally identified with $A$ by sending $B$ on $a$.
\end{enumerate}
The partial compositions of posets, which we denote by $\circ_a$, are then defined by:
\begin{equation}\label{poset-partial-composition}
\Cal A\circ_a\Cal B:=\sum_{\le\in\Omega(\Cal A,a,\Cal B)}(A\sqcup_a B,\le).
\end{equation}
For example, we have
$$\scalebox{1.2}{\tddeux{$a$}{$b$}}\circ_a\scalebox{1.2}{\tddeux{$1$}{$2$}}=\scalebox{1.2}{\tdtroisun {$1$}{$2$}{$b$}}+\scalebox{1.2}{\tdtroisdeux {$1$}{$2$}{$b$}}$$
and
$$\scalebox{1.2}{\tdtroisdeux{$a$}{$b$}{$c$}}\circ_b\scalebox{1.2}{\tddeux{$1$}{$2$}}=\scalebox{1.2}{\tdquatrecinq {$a$}{$1$}{$2$}{$c$}}+\scalebox{1.2}{\tdquatrequatre {$a$}{$1$}{$2$}{$c$}}+\ \scalebox{1.2}{\pdquatrequatre {$a$}{$1$}{$2$}{$c$}}+\ \scalebox{1.2}{\pdquatresept {$a$}{$1$}{$c$}{$2$}}+\scalebox{1.2}{\pdquatresix{$1$}{$a$}{$c$}{$2$}}\ .$$
Note that $a$ and $c$ are not comparable in the last term, but become comparable (namely $a\leqslant c$) when the subposet $\scalebox{1.2}{\tddeux{$1$}{$2$}}$ is shrunk.
\ignore{
\begin{lem}\label{poset-asso-par}
Let $\Cal A=(A,\le_A)$,  $\Cal B=(B,\le_B)$ and $\Cal C=(C,\le_C)$ be three finite posets, and let $a,a'\in A$. Then we have:
\begin{equation}\label{Omega-un}
\bigsqcup_{\le\in\Omega(\Cal A,a,\Cal B)}\Omega\big((A\sqcup_a B,\le),a',\Cal C\big)=
\bigsqcup_{\le\in\Omega(\Cal A,a',\Cal C)}\Omega\big((A\sqcup_{a'} C,\le),a,\Cal B\big).
\end{equation}
\end{lem}
\begin{proof}
\redtext{preuve \`a d\'etailler} Both members of \eqref{Omega-un} coincide with the set of all partial orders $\le$ on $(A\sqcup_a B)\sqcup_{a'} C$ such that:
\begin{enumerate}
\item $\le\restr{B}=\le_B$,
\item $\le\restr{C}=\le_C$,
\item $\big((A\sqcup_a B)\sqcup_{a'} C,\,\le\big)/(B\to\hskip -3mm\to \{a\}, C\to\hskip -3mm\to \{a'\})\sim \Cal A$.
\end{enumerate}
\end{proof}
\noindent The set defined by both members of \eqref{Omega-un} will be denoted by $\Omega(\Cal A, a,a',\Cal B,\Cal C)$.
\begin{lem}\label{poset-asso-emb}
Let $\Cal A=(A,\le_A)$,  $\Cal B=(B,\le_B)$ and $\Cal C=(C,\le_C)$ be three finite posets, and let $a\in A$, $b\in B$. Then we have:
\begin{equation}\label{Omega-deux}
\bigsqcup_{\le\in\Omega(\Cal B,b,\Cal C)}\Omega\big(\Cal A, a, (B\sqcup_b C,\le)\big)=
\bigsqcup_{\le\in\Omega(\Cal A,a,\Cal B)}\Omega\big((A\sqcup_{a} B,\le),b,\Cal C\big).
\end{equation}
\end{lem}
\begin{proof}
\redtext{preuve \`a d\'etailler} Both members of \eqref{Omega-deux} coincide with the set of all partial orders $\le$ on $A\sqcup_a B\sqcup_{b} C$ such that:
\begin{enumerate}
\item $\le\restr{C}=\le_C$,
\item $(B\sqcup_b C,\le\restr{B\sqcup_b C})/(C\to\hskip -3mm\to \{b\})\sim\Cal B$,
\item $\big(A\sqcup_a B\sqcup_{b} C,\,\le\big)/(B\sqcup_b C\to\hskip -3mm\to \{a\})\sim \Cal A$.
\end{enumerate}
\end{proof}
\noindent The set defined by both members of \eqref{Omega-deux} will be denoted by $\Omega(\Cal A, a,\Cal B,b,\Cal C)$.
}
\begin{thm}\label{poset-operad-main}
The partial compositions \eqref{poset-partial-composition} define an operad structure on the species of finite posets.
\end{thm}
\noindent We postpone the proof to Paragraph \ref{sect:proof} in the next section.
\section{The set triple operad of finite connected posets}
We will define set-theoretical partial compositions by retaining one privileged term in the sum \eqref{poset-partial-composition} defining the partial compositions $\circ_a$. It turns out that three different procedures are available leading to three families of partial compositions $\bullet_a$, $\blacktriangledown_a$ and $\blacktriangle_a$, providing the species of finite posets with three distinct set operad structures.\\

For any two finite posets $\Cal A=(A,\le_A)$ and $\Cal B=(B,\le_B)$, and $a\in A$, the partial order on $\Cal A\bullet_a\Cal B$ is obtained from $\le_A$ and $\le_B$ by saturation, and the operad axioms can be deduced from saturation by stages. We give an account of this general phenomenon in Paragraph \ref{sect:saturation} below before looking at the partial compositions $\bullet_a$ themselves.\\

The operad axioms for the family $\blacktriangledown_a$, and similarly for $\blacktriangle_a$, are obtained by a direct check, which can also be used with very few modifications for $\bullet_a$.
\subsection{Saturation of relations}\label{sect:saturation}
Let $E$ be a set. Let $\sim$ be an equivalence relation on $E$, let $\Cal R$ be a binary relation on $E/\sim$ and let $\Cal S$ be a binary relation on $E$. The class of $e\in E$ in $E/\sim$ will be denoted by $[e]$. Let $\Cal T$ be the relation on $E$ defined by $e\Cal T e'$ if and only if, either $[e]\neq [e']$ and $[e]\Cal R[e']$, or $[e]= [e']$ and $e\Cal S e'$.
\begin{prop}\label{saturation-orders}
If the two binary relations $\Cal R$ and $\Cal S$ above are partial orders, then $\Cal T$ is also a partial order.
\end{prop}
\begin{proof}
Reflexivity is obvious. Now let $e,e',e''\in E$ such that $e\Cal T e'$ and $e'\Cal T e''$. Several cases can occur:
\begin{enumerate}
\item If $[e]=[e']=[e'']$, then $e\Cal S e'$ and $e'\Cal S e''$, hence $e\Cal S e''$, which yields $e\Cal T e''$.
\item If $[e]=[e']\neq[e'']$, then $[e']\Cal R [e'']$. Hence $[e]\Cal R [e'']$, which yields $e\Cal T e''$. 
Similarly for $[e]\neq[e']=[e'']$.
\item If $[e]\neq[e']$ and $[e']\neq[e'']$, then $[e]\Cal R[e']$ and $[e']\Cal R[e'']$, hence $[e]\Cal R [e'']$.
Moreover, $[e]\neq [e'']$: otherwise, we would have $[e]=[e']=[e'']$, as $\Cal R$ is a partial order. This yields $e\Cal T e''$.
\end{enumerate}
This proves transitivity of $\Cal T$. Finally, if $e\Cal T e'$ and $e'\Cal T e$ then:
\begin{enumerate}
\item If $[e]=[e']$, then $e\Cal S e'$ and $e'\Cal S e$, hence $e=e'$ by antisymmetry of $\Cal S$.
\item If $[e]\neq[e']$, then $[e]\Cal R [e']$ and $[e']\Cal R [e]$, contradiction.
\end{enumerate}
This proves antisymmetry of $\Cal T$.
\end{proof}
\begin{prop}[Saturation by stages]\label{saturation-by-stages}
Let $E$ be a set. Let $\sim$ be an equivalence relation on $E$, let $\approx$ be an equivalence relation on $E$ finer than $\sim$, and let $\simeq$ be the equivalence relation on $E/\approx$ deduced from $\sim$. The quotient $E/\sim$ is in bijection with $(E/\approx)/\simeq$. Let
\begin{itemize}
\item $\Cal R$ be a binary relation on $E/\sim\,=\,(E/\approx)/\simeq$,
\item $\Cal S$ be a binary relation on $E/\approx$,
\item $\Cal T$ be the relation on $E/\approx$ obtained from $\Cal R$ and $\Cal S$ by saturation as explained above,
\item $\Cal S'$ be a binary relation on $E$,
\item $\Cal S''$ be the binary relation on $E$ obtained from $\Cal S$ and $\Cal S'$ by saturation,
\item $\Cal T'$ be the relation on $E$ obtained from $\Cal T$ and $\Cal S'$ by saturation.
\end{itemize}
Then $\Cal T'$ is also obtained from $\Cal R$ and $\Cal S''$ by saturation.
\end{prop}
\begin{proof}
For any $e\in E$, we will denote by $[e]$ the class of $e$ for the relation $\approx$, and by $[[e]]$ the class of $e$ for $\sim$, which is the same as the class of $[e]$ for $\simeq$. We have $e\Cal T' e'$ if and only if either $[e]=[e']$ and $e\Cal S' e'$, or $[e]\neq[e']$ and $[e]\Cal T [e']$. Using the definition of $\Cal T$ we have that $e\Cal T' e'$ if and only if:
\begin{itemize}
\item either $[e]=[e']$ and $e\Cal S' e'$,
\item or $[e]\neq[e']$, $[[e]]=[[e']]$ and $[e]\Cal S [e']$,
\item or $[[e]]\neq [[e']]$ and $[[e]]\Cal R [[e']]$.
\end{itemize}
Hence $e\Cal T' e'$ if and only if:
\begin{itemize}
\item either $[[e]]=[[e']]$ and $e\Cal S'' e'$,
\item or $[[e]]\neq [[e']]$ and $[[e]]\Cal R [[e']]$,
\end{itemize}
which proves the claim.
\end{proof}
\subsection{The first set operad structure}
Let $\Cal A=(A,\le_A)$ and $\Cal B=(B,\le_B)$ be two finite posets, and let $a\in A$. Then we define:
\begin{equation}
\Cal A\bullet_a\Cal B:=(A\sqcup_a B,\le),
\end{equation}
where $\le$ is obtained from $\le_A$ on $A=A\sqcup_a B/\sim$ and $\le_B$ on $A\sqcup B$ by the saturation process described in Paragraph \ref{sect:saturation}. Here the equivalence relation $\sim$ is obtained by identifying all elements of $B$ with $a$, and the partial order $\le_B$ is trivially extended to $A\sqcup_a B$, i.e. $x\le y$ if and only if $x=y$ in case $x$ or $y$ does not belong to $B$. The relation $\le$ is a partial order by Proposition \ref{saturation-orders}. To sum up, for any $x,y\in A\sqcup_a B$, $x\le y$ if and only if:
\begin{itemize}
\item either $x,y\in B$ and $x\le_B y$,
\item or $x,y\in A\setminus\{a\}$ and $x\le_A y$,
\item or $x\in A\setminus\{a\}$, $y\in B$ and $x\le_A a$,
\item or $x\in B$, $y\in A\setminus\{a\}$ and $a\le_A y$.
\end{itemize}
for exemple, we have:
\begin{align*}
\scalebox{1.2}{\tdtroisdeux{$a$}{$b$}{$c$}}\bullet_b\scalebox{1.2}{\tddeux{$1$}{$2$}}&=\scalebox{1.2}{\tdquatrecinq {$a$}{$1$}{$2$}{$c$}},\\
\scalebox{1.2}{\tdtroisdeux{$a$}{$b$}{$c$}}\bullet_b\scalebox{1.2}{\tdun{$1$}\tdun{$2$}}&=\scalebox{1.2}{\pdquatrehuit{$a$}{$c$}{$1$}{$2$}}.
\end{align*}
\begin{rmk}
Note that the partial order $\le$ defined above is an element of the set $\Omega(A,a,B)$ defined in Paragraph \ref{sect:poset-operad}.
\end{rmk}
\begin{thm}\label{poset-operad-set}
The partial compositions $\bullet_a$ endow the set species of finite posets with a structure of operad.
\end{thm}
\begin{proof}
Partial composition of two posets is a single poset, and functoriality with respect to relabelling is obvious. It remains to check the two associativity axioms. The nested associativity axiom is directly derived from Proposition \ref{saturation-by-stages}, with $E=A\sqcup_b B\sqcup_b C$. The equivalence relation $\approx$(resp. $\sim$) is obtained by shrinking $C$ to $\{b\}$ (resp. $B\sqcup_b C$ to $\{a\}$). The set $E/\approx$ is identified with $A\sqcup_a B$, and $\simeq$ is obtained by shrinking $B$ on $\{a\}$. It remains to apply Proposition \ref{saturation-by-stages} with the following dictionary:
\begin{itemize}
\item $\Cal R$: partial order $\le_A$ on $A=E/\sim$.
\item $\Cal S$: partial order $\le_B$ trivially extended on $E/\approx\,=A\sqcup_a B$.
\item $\Cal T$: partial order of $\Cal A\bullet_b\Cal B$ on $E/\approx\,=A\sqcup_a B$.
\item $\Cal S'$: partial order $\le_C$ trivially extended to $E$,
\item $\Cal S''$: partial order of $\Cal B\bullet_b\Cal C$ trivially extended from $B\sqcup_b C$ to $E$. 
\item $\Cal T'$: partial order of $(\Cal A\bullet_a\Cal B)\bullet_b\Cal C$.
\end{itemize}
\noindent The proof of the parallel associativity is similar and left to the reader.
\end{proof}
\subsection{Two more set operad structures}
We define two families of partial compositions $\blacktriangle_a$ and $\blacktriangledown_a$ as follows: for two finite posets $\Cal A=(A,\le_A)$ and $\Cal B=(B,\le_B)$ and $a\in A$, the poset $\Cal A\blacktriangle_a\Cal B$ is the set $A\sqcup_a B$ together with the partial order $\le$ defined by $x\le y$ if and only if:
\begin{itemize}
\item either $x,y\in B$ and $x\le_B y$,
\item or $x,y\in A\setminus\{a\}$ and $x\le_A y$,
\item or $x\in A\setminus\{a\}$, $y\in \mop{max}B$ and $x\le_A a$,
\item or $x\in B$, $y\in A\setminus\{a\}$ and $a\le_A y$.
\end{itemize}
Similarly, the poset $\Cal A\blacktriangledown_a\Cal B$ is the set $A\sqcup_a B$ together with the partial order $\le$ defined by $x\le y$ if and only if:
\begin{itemize}
\item either $x,y\in B$ and $x\le_B y$,
\item or $x,y\in A\setminus\{a\}$ and $x\le_A y$,
\item or $x\in A\setminus\{a\}$, $y\in B$ and $x\le_A a$,
\item or $x\in \mop{min}B$, $y\in A\setminus\{a\}$ and $a\le_A y$.
\end{itemize}
\begin{thm}
Both families of partial compositions $\blacktriangle_a$ and $\blacktriangledown_a$ endow the species of finite posets with a set operad structure.
\end{thm}
\begin{proof}
Let us prove it for the family $\blacktriangledown_a$. The proof for $\blacktriangle_a$ is entirely similar. Let $\Cal A=(A,\le_A)$, $\Cal B=(B,\le_B)$ and $\Cal C=(C,\le_C)$ be three finite posets, with $a\in A$ and $b\in B$. The poset $(\Cal A \blacktriangledown_a \Cal B)\blacktriangledown_b\Cal C$ is the set $A\sqcup_a B\sqcup_b C$ endowed with the partial order $\le$ defined by: $x\le y$ if and only if:
\begin{itemize}
\item either $x,y\in C$ and $x\le_C y$,
\item or $x,y\in A\sqcup_a B\setminus\{b\}$ and $x\le_{\Cal A \blacktriangledown_a \Cal B}y$,
\item or $x\in A\sqcup_a B\setminus\{b\}$, $y\in C$ and $x\le_{\Cal A \blacktriangledown_a \Cal B} b$,
\item or $x\in \mop{min}\Cal C$, $y\in A\sqcup_a B\setminus\{b\}$ and $b\le_{\Cal A \blacktriangledown_a \Cal B} y$.
\end{itemize}
\vskip 2mm
This in turn expands into: $x\le y$ if and only if:
\vskip 2mm
\begin{itemize}
\item (1) either $x,y\in C$ and $x\le_C y$,
\vskip 2mm
\item (2) or $x,y\in B\setminus\{b\}$ and $x\le_B y$,
\item (3) or $x,y\in A\setminus\{a\}$ and $x\le_A y$,
\item (4) or $x\in A\setminus\{a\}$, $y\in B\setminus \{b\}$ and $x\le_A a$,
\item (5) or $x\in (\mop{min} \Cal B)\setminus\{b\}$, $y\in A\setminus\{a\}$ and $a\le_A y$,
\vskip 2mm
\item (6) or $x\in B\setminus\{b\}$, $y\in C$ and $x\le_B b$,
\item (7) or $x\in A\setminus\{a\}$, $y\in C$ and $x\le_A a$,
\vskip 2mm
\item (8) or $x\in \mop{min}\Cal C$, $y\in B\setminus\{b\}$ and $b\le_B y$,
\item (9) or $x\in \mop{min}\Cal C$, $y\in A\setminus\{a\}$, $b\in\mop{min}\Cal B$ and $a\le_A y$.
\end{itemize}
On the other hand, the poset $\Cal A \blacktriangledown_a (\Cal B\blacktriangledown_b\Cal C)$ is the set $A\sqcup_a B\sqcup_b C$ endowed with the partial order $\leqslant$ defined by: $x\leqslant y$ if and only if:
\begin{itemize}
\item either $x,y\in B\sqcup_b C$ and $x\le_{\Cal B\blacktriangledown_b\Cal C} y$,
\item or $x,y\in A\setminus\{a\}$ and $x\le_A y$,
\item or $x\in A\setminus\{a\}$, $y\in B\sqcup_b C$ and $x\le_A a$,
\item or $x\in \mop{min}\Cal B\blacktriangledown_b\Cal C $, $y\in A\setminus\{a\}$ and $a\le_A y$.
\end{itemize}
\vskip 2mm
This in turn expands into: $x\leqslant y$ if and only if:
\vskip 2mm
\begin{itemize}
\item (1) either $x,y\in C$ and $x\le_C y$,
\item (2) or $x,y\in B\setminus\{b\}$ and $x\le_B y$,
\item (6) or $x\in B\setminus\{b\}$, $y\in C$ and $x\le_B b$,
\item (8) or $x\in\mop{min}\Cal C$, $y\in B\setminus\{b\}$ and $b\le_B y$,
\vskip 2mm
\item (3) or $x,y\in A\setminus\{a\}$ and $x\le_A y$,
\vskip 2mm
\item (4) or $x\in A\setminus\{a\}$, $y\in B\setminus\{b\}$ and $x\le_A a$,
\item (7) or $x\in A\setminus\{a\}$, $y\in C$ and $x\le_A a$,
\vskip 2mm
\item (9) or $x\in\mop{min}\Cal C$, $b\in\mop{min}\Cal B$, $y\in A\setminus\{a\}$ and $a\le_a y$,
\item (5) or $x\in (\mop{min}\Cal B)\setminus\{b\}$, $y\in A\setminus\{a\}$ and $a\le_A y$.
\end{itemize}
Hence the two partial orders $\le$ and $\leqslant$ coincide of $A\sqcup_a B\sqcup_b C$, which proves the nested associativity axiom. Now let $a'\neq a$ be a second element of $A$. The poset $(\Cal A \blacktriangledown_a \Cal B)\blacktriangledown_{a'}\Cal C$ is the set $A\sqcup B\sqcup C\setminus\{a,a'\}$ endowed with the partial order $\le$ defined by: $x\le y$ if and only if:
\begin{itemize}
\item either $x,y\in C$ and $x\le_C y$,
\item or $x,y\in A\sqcup_a B\setminus\{a'\}$ and $x\le_{\Cal A \blacktriangledown_a \Cal B}y$,
\item or $x\in A\sqcup_a B\setminus\{a'\}$, $y\in C$ and $x\le_{\Cal A \blacktriangledown_a \Cal B} a'$,
\item or $x\in \mop{min}\Cal C$, $y\in A\sqcup_a B\setminus\{a'\}$ and $a'\le_{\Cal A \blacktriangledown_a \Cal B} y$.
\end{itemize}
\vskip 2mm
This in turn expands into: $x\le y$ if and only if:
\vskip 2mm
\begin{itemize}
\item either $x,y\in C$ and $x\le_C y$,
\vskip 2mm
\item or $x,y\in B$ and $x\le_B y$,
\item or $x,y\in A\setminus\{a,a'\}$ and $x\le_A y$,
\item or $x\in A\setminus\{a,a'\}$, $y\in B$ and $x\le_A a$,
\item or $x\in \mop{min}\Cal B$, $y\in A\setminus\{a,a'\}$ and $a\le_A y$,
\vskip 2mm
\item or $x\in \mop{min}\Cal B$, $y\in C$ and $a\le_A a'$,
\item or $x\in A\setminus\{a,a'\}$, $y\in C$ and $x\le_A a'$,
\vskip 2mm
\item or $x\in \mop{min}\Cal C$, $y\in B$ and $a'\le_A a$,
\item or $x\in \mop{min}\Cal C$, $y\in A\setminus\{a,a'\}$ and $a'\le_A y$.
\end{itemize}
Exchanging $(B,a)$ with $(C,a')$ leaves the nine conditions above globally unchanged, which proves the parallel associativity axiom.
\end{proof}
\begin{rmk}
{\rm The same proof can be written almost word for word for the operad $(\mathbb P,\bullet)$: it suffices to replace $\blacktriangledown_a$ and $\blacktriangledown_b$ by $\bullet_a$ and $\bullet_b$ respectively, and to suppress "min" everywhere in the proof. Hence we have two proofs for Theorem \ref{poset-operad-set}.}
\end{rmk}
We can see the difference between the three structures in this example:
\begin{align*}
\scalebox{1.2}{\tdtroisdeux{$a$}{$b$}{$c$}}\bullet_b\scalebox{1.2}{\tddeux{$1$}{$2$}}&=\ \scalebox{1.2}{\tdquatrecinq {$a$}{$1$}{$2$}{$c$}},\\
\scalebox{1.2}{\tdtroisdeux{$a$}{$b$}{$c$}}\blacktriangledown_b\scalebox{1.2}{\tddeux{$1$}{$2$}}&=\scalebox{1.2}{\tdquatrequatre {$a$}{$1$}{$2$}{$c$}},\\
\scalebox{1.2}{\tdtroisdeux{$a$}{$b$}{$c$}}\blacktriangle_b\scalebox{1.2}{\tddeux{$1$}{$2$}}&=\scalebox{1.2}{\pdquatrequatre {$a$}{$1$}{$2$}{$c$}}\ .
\end{align*}
\begin{rmk}
\rm The family of partial compositions $\blacktriangledown_a$ restricts itself to the rooted trees. The suboperad thus defined is nothing but the NAP operad \cite{L06}. The pre-Lie operad \cite{CL01} is however \textsl{not} obtained form the family of partial compositions $\circ_a$, which don't restrict themselves to the rooted trees.
\end{rmk}
\subsection{Proof of Theorem \ref{poset-operad-main}}\label{sect:proof}
Recall the species automorphism $\Phi$ of $\mathbb P$ described in the Introduction, defined by:
$$\Phi(\Cal A)=\sum_{\Cal A'\preceq\Cal A}\Cal A'.$$
Let $\Cal A=(A,\le_A)$ and $\Cal B=(B,\le_B)$ be two finite posets, and let $a\in A$. Let us introduce the two following sets: $\Omega'(\Cal A,a,\Cal B)$ is the set of all partial orders $\le$ on $A\sqcup_a B$ such that
\begin{itemize}
\item $B$ is convex in $(A\sqcup_a B,\le)$,
\item $(B,\le\restr{B})\preceq\Cal B$,
\item $(A\sqcup_a B,\le)/B\preceq \Cal A$, where the set $(A\sqcup_a B)/B=(A\setminus\{a\})\sqcup \{B\}$ is naturally identified with $A$ by sending $B$ on $a$,
\end{itemize}
and $\Omega''(\Cal A,a,\Cal B)$ is the set of all partial orders $\le$ on $A\sqcup_a B$ such that
$$(A\sqcup_aB,\le)\preceq \Cal A\bullet_a\Cal B.$$
We obviously have:
\begin{equation}
\Phi(\Cal A\bullet_a\Cal B)=\sum_{\le\in\Omega''(\Cal A,a,\Cal B)}(A\sqcup_aB,\le)
\end{equation}
and
\begin{equation}
\Phi(\Cal A)\circ_a\Phi(\Cal B)=\sum_{\le\in\Omega'(\Cal A,a,\Cal B)}(A\sqcup_aB,\le).
\end{equation}
To prove Theorem \ref{poset-operad-main}, it is then enough to prove the equality $\Omega'(\Cal A,a,\Cal B)=\Omega''(\Cal A,a,\Cal B)$. Let us first prove the inclusion $\Omega'(\Cal A,a,\Cal B)\subset\Omega''(\Cal A,a,\Cal B)$: for any partial order $\le\in \Omega'(\Cal A,a,\Cal B)$ and $x,y$ such that $x\le y$, we have four cases to look at:
\begin{itemize}
\item if $x,y\in B$ then $x\le\restr{B}\:y$, hence $x\le_B y$, hence $x\le_{\Cal A\bullet_a\Cal B}y$.
\item If $x\in A\setminus\{a\}$ and $y\in B$, then $x\le_{ A}a$, hence $x\le_{\Cal A\bullet_a\Cal B}y$.
\item If $x\in B$ and $y\in A\setminus \{a\}$, then $a\le_{A}y$, hence $x\le_{\Cal A\bullet_a\Cal B}y$.
\item If $x,y\in A\setminus \{a\}$, then $x\le_{A}y$, hence $x\le_{\Cal A\bullet_a\Cal B}y$.
\end{itemize}
To sum up, for any partial order $\le\in \Omega'(\Cal A,a,\Cal B)$ and for any $x,y\in A\sqcup_a B$, $x\le y\Rightarrow x\le_{\Cal A\bullet_a\Cal B}y$, which shows the inclusion.\\

Let us now prove the reverse inclusion $\Omega''(\Cal A,a,\Cal B)\subset\Omega'(\Cal A,a,\Cal B)$: first of all, for any partial order $\le\in \Omega''(\Cal A,a,\Cal B)$ and for any $x,z,y\in A\sqcup_a B$ with $x,y\in B$ and $x\le z\le y$, we have $x\le_{\Cal A\bullet_a\Cal B} z\le_{\Cal A\bullet_a\Cal B} y$, hence $z\in B$ by convexity of $B$ in $\Cal A\bullet_a\Cal B$. Hence $B$ is convex in $(A\sqcup_a B,\le)$. Now for any $x,y\in B$ we have $x\le y\Rightarrow x\le_{\Cal A\bullet_a\Cal B} y\Rightarrow x\le_B y$. Now let us denote by $\leqslant$ the partial order on $A\sim (A\sqcup_a B)/B$ corresponding to the poset $(A\sqcup_a B,\le)/B$. We just have to prove that for any $x,y\in A$, $x\leqslant y\Rightarrow x\le_A y$. Four cases can occur:
\begin{enumerate}
\item if $x,y\in A\setminus\{a\}$, two subcases occur:
\begin{itemize}
\item if $x\le_A y$ then $x\le_{\Cal A\bullet_a\Cal B} y$, then $x\le_A y$.
\item If there exist $b,b'\in B$ such that $x\le b$ and $b'\le y$, then $x\le_{\Cal A\bullet_a\Cal B} b$ and $b'\le_{\Cal A\bullet_a\Cal B} y$, hence $x\le_{A}a$ and $a\le_{A}y$, which yields $x\le_A y$.
\end{itemize}
\item If $x\in A\setminus\{a\}$ and $y=a$, there exists $b\in B$ such that $x\le b$, hence $x\le_{\Cal A\bullet_a\Cal B} b$, hence $x\le_A y$.
\item The case $x=a$ and $y\in A\setminus\{a\}$ is treated similarly.
\item The case $x=y=a$ is trivial.
\end{enumerate}
This proves the equality $\Omega'(\Cal A,a,\Cal B)=\Omega''(\Cal A,a,\Cal B)$, which in turn proves Theorem \ref{poset-operad-main}.
\begin{cor}
The partial compositions $\circ$ endow the species $\mathbb P$ with a linear operad structure, isomorphic to $(\mathbb P,\bullet)$.
\end{cor}
\subsection{The triple suboperad generated by the connected poset with two elements}
It is possible to generate all connected posets up to four vertices with the poset $\scalebox{1.3}{\tdeux}$ and the three families of partial compositions $\bullet$, $\blacktriangledown$ and $\blacktriangle$. Indeed,
\begin{align*}
\scalebox{1.2}{\tdtroisdeux{$a$}{$b$}{$c$}}&=\scalebox{1.2}{\tddeux{$a$}{$1$}}\bullet_1\scalebox{1.2}{\tddeux{$b$}{$c$}},\hskip 12mm
\scalebox{1.2}{\tdtroisun{$a$}{$c$}{$b$}}=\scalebox{1.2}{\tddeux{$1$}{$b$}}\blacktriangledown_1\scalebox{1.2}{\tddeux{$a$}{$c$}},\hskip 12mm
\scalebox{1.2}{\pdtroisun{$a$}{$b$}{$c$}}=\scalebox{1.2}{\tddeux{$b$}{$1$}}\blacktriangle_1\scalebox{1.2}{\tddeux{$c$}{$a$}},\\
\scalebox{1.2}{\tdquatrecinq{$a$}{$b$}{$c$}{$d$}}&=\scalebox{1.2}{\tdtroisdeux{$1$}{$c$}{$d$}}\bullet_1\scalebox{1.2}{\tddeux{$a$}{$b$}},\hskip 12mm
\scalebox{1.2}{\tdquatrequatre{$a$}{$b$}{$d$}{$c$}}=\scalebox{1.2}{\tdtroisdeux{$a$}{$1$}{$c$}}\blacktriangledown_1\scalebox{1.2}{\tddeux{$b$}{$d$}},\hskip 12mm
\scalebox{1.2}{\tdquatredeux{$a$}{$d$}{$b$}{$c$}}=\scalebox{1.2}{\tdtroisdeux{$1$}{$b$}{$c$}}\blacktriangledown_1\scalebox{1.2}{\tddeux{$a$}{$d$}},\\
\scalebox{1.2}{\tdquatreun{$a$}{$d$}{$c$}{$b$}}&=\scalebox{1.2}{\tdtroisun{$1$}{$c$}{$b$}}\blacktriangledown_1\scalebox{1.2}{\tddeux{$a$}{$d$}},\hskip 12mm 
\scalebox{1.2}{\pdquatrequatre{$a$}{$b$}{$c$}{$d$}}=\scalebox{1.2}{\tdtroisdeux{$a$}{$1$}{$d$}}\blacktriangle_1\scalebox{1.2}{\tddeux{$b$}{$c$}},\hskip 12mm \scalebox{1.2}{\pdquatredeux{$a$}{$b$}{$c$}{$d$}}=\scalebox{1.2}{\tdtroisdeux{$a$}{$b$}{$1$}}\blacktriangle_1\scalebox{1.2}{\tddeux{$c$}{$d$}},\\
\scalebox{1.2}{\pdquatreun{$a$}{$d$}{$c$}{$b$}}&=\scalebox{1.2}{\pdtroisun{$1$}{$a$}{$b$}}\blacktriangle_1\scalebox{1.2}{\tddeux{$c$}{$d$}},\hskip 12mm 
\scalebox{1.2}{\pdquatresix{$a$}{$d$}{$b$}{$c$}}=\scalebox{1.2}{\tdtroisun{$a$}{$1$}{$b$}}\blacktriangle _1\scalebox{1.2}{\tddeux{$d$}{$c$}},\hskip 12mm
\scalebox{1.2}{\pdquatresept{$a$}{$d$}{$b$}{$c$}}=\scalebox{1.2}{\tddeux{$1$}{$b$}}\blacktriangledown_1\scalebox{1.2}{\pdtroisun{$c$}{$a$}{$d$}},\\
&\\
\scalebox{1.2}{\pdquatrehuit{$a$}{$d$}{$b$}{$c$}}&=\scalebox{1.2}{\tddeux{$a$}{$1$}}\bullet_1\scalebox{1.2}{\pdtroisun{$b$}{$c$}{$d$}}.
\end{align*}
One cannot generate all finite posets with the poset $\scalebox{1.3}{\tdeux}$ and the three families of partial compositions $\bullet$, $\blacktriangledown$ and $\blacktriangle$. For example, $\psix$\hskip 4mm cannot be reached this way.

\section{Compatibilities for the operadic products}\label{sect:comp}

\begin{prop}\label{compatibilities}
Let $\mathcal{A}=(A,\leq_A)$, $\mathcal{B}=(B,\leq_B)$ and $\mathcal{C}=(C,\leq_C)$ be three finite posets, and let $a,b \in A$, distinct. Then:
\begin{align*}
(\mathcal{A}\blacktriangle_a \mathcal{B})\bullet_b \mathcal{C}&=(\mathcal{A}\bullet_b \mathcal{C})\blacktriangle_a \mathcal{B},\\
(\mathcal{A}\blacktriangledown_a \mathcal{B})\bullet_b \mathcal{C}&=(\mathcal{A}\bullet_b \mathcal{C})\blacktriangledown_a \mathcal{B},\\
(\mathcal{A}\blacktriangledown_a \mathcal{B})\blacktriangle_b \mathcal{C}&=(\mathcal{A}\blacktriangle_b \mathcal{C})\blacktriangledown_a \mathcal{B}.
\end{align*}\end{prop}

\begin{proof} 1. We put $(\mathcal{A}\blacktriangle_a \mathcal{B})\bullet_b \mathcal{C}=(A\sqcup_a B\sqcup_b C,\leq_S)$
 and $(\mathcal{A}\bullet_b \mathcal{C})\blacktriangle_a \mathcal{B}=(A\sqcup_b C\sqcup_a B,\leq_T) $.
Let $x\in A\sqcup_a B \sqcup_b C$. Then $x\leq_S y$ if:
$$\begin{array}{c|c|c|c}
&y\in A\setminus\{a,b\}&y\in B&y\in C\\
\hline x\in A\setminus\{a,b\}&x\leq_A y&y\in \max(\mathcal{B})\mbox{ and } x\leq_A a&x\leq_A b\\
\hline x\in B&a\leq_A y&x\leq_B y&x\leq_{\mathcal{A}\blacktriangle_a \mathcal{B}}b\Longleftrightarrow a\leq_A b\\
\hline x\in C&b\leq_A y&b\leq_{\mathcal{A}\blacktriangle_a \mathcal{B}}y \Longleftrightarrow&x\leq_C y\\
&&y\in \max(\mathcal{B})\mbox{ and }b\leq_A a&
\end{array}$$
Moreover, $x\leq_T y$ if:
$$\begin{array}{c|c|c|c}
&y\in A\setminus\{a,b\}&y\in B&y\in C\\
\hline x\in A\setminus\{a,b\}&x\leq_A y&y\in \max(\mathcal{B})\mbox{ and } x\leq_A a&x\leq_A b\\
\hline x\in B&a\leq_A y&x\leq_B y&x\leq_{\mathcal{A}\bullet_b \mathcal{C}}b\Longleftrightarrow a\leq_A b\\
\hline x\in C&b\leq_A y&y\in \max(\mathcal{B}) \mbox{ and } x\leq_{\mathcal{A}\bullet_b \mathcal{C}}a\Longleftrightarrow&x\leq_C y\\
&&y\in \max(\mathcal{B})\mbox{ and }b\leq_A a&
\end{array}$$
So $\le_S=\le_T$.\\

2. Can be deduced from the first point, with the help of the involution on posets.\\

3. We put $(\mathcal{A}\blacktriangledown_a \mathcal{B})\blacktriangle_b \mathcal{C}=(A\sqcup_a B\sqcup C,\leq_S)$ and
$(\mathcal{A}\blacktriangle_b \mathcal{C})\blacktriangledown_a \mathcal{B}=(A\sqcup_b C\sqcup_a B,\leq_T)$.
Let $x,y\in A\sqcup_a B \sqcup_b C$. Then $x\leq_S y$ if:
$$\begin{array}{c|c|c|c}
&y\in A\setminus\{a,b\}&y\in B&y\in C\\
\hline x\in A\setminus\{a,b\}&x\leq_A y&x\leq_A a&y\in \max(\mathcal{C}) \mbox{ and }x\leq_A a\\
\hline x\in B&x\in \min(\mathcal{B})&x\leq_B y&y\in \max(\mathcal{C}),\: x\in \min(\mathcal{B})\mbox{ and }a\leq_A b\\
&\mbox{ and }a\leq_A y&&\\
\hline x\in C&b\leq_A y&b\leq_{\mathcal{A}\blacktriangledown a \mathcal{B}} y\Longleftrightarrow&x\leq_Cy\\
&&b\leq_A a&
\end{array}$$
Moreover, $x\leq_T y$ if:
$$\begin{array}{c|c|c|c}
&y\in A\setminus\{a,b\}&y\in B&y\in C\\
\hline x\in A\setminus\{a,b\}&x\leq_A y&x\leq_A a&y\in \max(\mathcal{C}) \mbox{ and }x\leq_A a\\
\hline x\in B&x\in \min(\mathcal{B})&x\leq_B y&y\in \max(\mathcal{C}),\: x\in \min(\mathcal{B})\mbox{ and }
a\leq_{\mathcal{A}\blacktriangle_b \mathcal{C}} y\Longleftrightarrow\\
&\mbox{ and }a\leq_A y&&y\in \max(\mathcal{C}),\: x\in \min(\mathcal{B})\mbox{ and }a\leq_A b\\
\hline x\in C&b\leq_A y&x\leq_{\mathcal{A}\blacktriangle_b \mathcal{C}} a\Longleftrightarrow&x\leq_Cy\\
&&b\leq_A a&
\end{array}$$
Hence $\le_S=\le_T$.
\end{proof}

\begin{rmk} Let $a\in A$ and $b\in B$. In general:
\begin{align*}
(\mathcal{A}\blacktriangle_a \mathcal{B})\bullet_b \mathcal{C}&\neq \mathcal{A}\blacktriangle_a (\mathcal{B}\bullet_b \mathcal{C}),&
(\mathcal{A}\blacktriangledown_a \mathcal{B})\bullet_b \mathcal{C}&\neq \mathcal{A}\blacktriangledown_a (\mathcal{B}\bullet_b \mathcal{C}),\\
(\mathcal{A}\bullet_a \mathcal{B})\blacktriangle_b \mathcal{C}&\neq \mathcal{A}\bullet_a (\mathcal{B}\blacktriangle_b \mathcal{C}),&
(\mathcal{A}\blacktriangledown_a \mathcal{B})\blacktriangle_b \mathcal{C}&\neq \mathcal{A}\blacktriangledown_a (\mathcal{B}\blacktriangle_b \mathcal{C}),\\
(\mathcal{A}\bullet_a \mathcal{B})\blacktriangledown_b \mathcal{C}&\neq \mathcal{A}\bullet_a (\mathcal{B}\blacktriangledown_b \mathcal{C}),&
(\mathcal{A}\blacktriangle_a \mathcal{B})\blacktriangledown_b \mathcal{C}&\neq \mathcal{A}\blacktriangle_a (\mathcal{B}\blacktriangledown_b \mathcal{C}).
\end{align*}
For example:
\begin{align*}
(\tddeux{}{$a$}\blacktriangle_a \tddeux{}{$b$})\bullet_b \tdeux&=\pdtroisun{$b$}{}{}\bullet_b \tdeux=\pquatrequatre,&
\tddeux{}{$a$}\blacktriangle_a (\tddeux{}{$b$}\bullet_b \tdeux)&=\tddeux{}{$a$}\blacktriangle_a\ttroisdeux=\pquatredeux.\\
(\tddeux{$a$}{}\blacktriangledown_a \tddeux{$b$}{})\bullet_b \tdeux&=\tdtroisun{$b$}{}{}\bullet_b \tdeux=\tquatrequatre,&
\tddeux{$a$}{}\blacktriangledown_a (\tddeux{$b$}{}\bullet_b \tdeux)&=\tddeux{$a$}{}\blacktriangledown_a\ttroisdeux=\tquatredeux.\\
(\tddeux{}{$a$}\bullet_a \tddeux{}{$b$})\blacktriangle_b \tdeux&=\tdtroisdeux{}{}{$b$}\bullet_b \tdeux=\pquatredeux,&
\tddeux{}{$a$}\bullet_a (\tddeux{}{$b$}\blacktriangle_b \tdeux)&=\tddeux{}{$a$}\blacktriangle_a\ptroisun=\pquatrehuit.\\
(\tddeux{$a$}{}\blacktriangledown_a \tddeux{}{$b$})\blacktriangle_b \tdeux&=\tdtroisun{}{}{$b$}\blacktriangle_b \tdeux=\pquatresix,&
\tddeux{$a$}{}\blacktriangledown_a (\tddeux{}{$b$}\blacktriangle_b \tdeux)&=\tddeux{$a$}{}\blacktriangle_a \ttroisun=\pquatresept.\\
(\tddeux{$a$}{}\bullet_a \tddeux{$b$}{})\blacktriangledown_b \tdeux&=\tdtroisdeux{$b$}{}{}\bullet_b \tdeux=\tquatredeux,&
\tddeux{$a$}{}\bullet_a (\tddeux{$b$}{}\blacktriangledown_b \tdeux)&=\tddeux{$a$}{}\blacktriangle_a\ttroisun=\pquatrehuit.\\
(\tddeux{}{$a$}\blacktriangle_a \tddeux{$b$}{})\blacktriangledown_b \tdeux&=\pdtroisun{}{$b$}{}\blacktriangledown_b \tdeux=\pquatrecinq,&
\tddeux{}{$a$}\blacktriangle_a (\tddeux{$b$}{}\blacktriangledown_b \tdeux)&=\tddeux{}{$a$}\blacktriangle_a \ttroisun=\pquatresept.
\end{align*}\end{rmk}

\section{Algebraic structures associated to these operads}\label{sect:alg}

\subsection{Products}

Let us introduce some notations: let $\mathcal{A}\in \mathbb{P}_{\{1,2\}}$ and $\mathcal{B}$, $\mathcal{C}$ be two finite posets. We put:
\begin{align*}
\mathcal{A}\circ(\mathcal{B},\mathcal{C})&=(\mathcal{A}\circ_1 \mathcal{B})\circ_2 \mathcal{C}
=(\mathcal{A}\circ_2 \mathcal{C}) \circ_1 \mathcal{B},\\
\mathcal{A}\bullet(\mathcal{B},\mathcal{C})&=(\mathcal{A}\bullet_1 \mathcal{B})\bullet_2 \mathcal{C}
=(\mathcal{A}\bullet_2 \mathcal{C}) \bullet_1 \mathcal{B},\\
\mathcal{A}\blacktriangle(\mathcal{B},\mathcal{C})&=(\mathcal{A}\blacktriangle_1 \mathcal{B})\blacktriangle_2 \mathcal{C}
=(\mathcal{A}\blacktriangle_2 \mathcal{C}) \blacktriangle_1 \mathcal{B},\\
\mathcal{A}\blacktriangledown(\mathcal{B},\mathcal{C})&=(\mathcal{A}\blacktriangledown_1 \mathcal{B})\blacktriangledown_2 \mathcal{C}
=(\mathcal{A}\blacktriangledown_2 \mathcal{C}) \blacktriangledown_1 \mathcal{B}.
\end{align*}

The free algebra on one generator over the different operadic structures on $\mathbb{P}$ is the vector space:
$$F_{\mathbb{P}}(1)=\bigoplus_{n=1}^\infty \mathbb{P}(\{1,\ldots,n\})/\mathfrak{S}_n,$$
so it can be identified with the vector space generated by the isomorphism classes (shortly, isoclasses) of finite posets.
This space inherits several bilinear products from the operad structures on posets, described below.\\

{\bf Notations.} If $\mathcal{A}$ is a finite poset, we denote by $\isoclasse{\mathcal{A}}$ its isomorphism class.

\begin{thm} \label{theoproduits}
Let $\mathcal{A}=(A,\leq_A)$ and $\mathcal{B}=(B,\leq_B)$ be two finite posets. Then we have in $F_{\mathbb{P}}(1)$:
\begin{enumerate}
\item $\tdun{$1$}\tdun{$2$}\circ(\isoclasse{\mathcal{A}},\isoclasse{\mathcal{B}})=
\tdun{$1$}\tdun{$2$}\bullet(\isoclasse{\mathcal{A}},\isoclasse{\mathcal{B}})=
\tdun{$1$}\tdun{$2$}\blacktriangle(\isoclasse{\mathcal{A}},\isoclasse{\mathcal{B}})=
\tdun{$1$}\tdun{$2$}\blacktriangledown(\isoclasse{\mathcal{A}},\isoclasse{\mathcal{B}})=\isoclasse{\mathcal{A}\mathcal{B}}$,
where $\mathcal{A}\mathcal{B}=(A\sqcup B,\leq)$, with, for all $x,y \in A\sqcup B$; $x\leq y$ if and only if:
\begin{itemize}
\item $x,y\in A$ and $x\leq_A y$,
\item or $x,y \in B$ and $x\leq_B y$.
\end{itemize} 
This product, defined on isoclasses of posets and also denoted by $m$, is sometimes called the disjoint union \cite{Stanley}.
It is associative and commutative.
\item $\tddeux{$1$}{$2$}\bullet(\isoclasse{\mathcal{A}},\isoclasse{\mathcal{B}})=\isoclasse{\mathcal{A}\downarrow \mathcal{B}}$, 
where $\mathcal{A}\downarrow \mathcal{B}=(A\sqcup B,\leq)$, with, for all $x,y \in A\sqcup B$; $x\leq y$ if and only if:
\begin{itemize}
\item $x,y\in A$ and $x\leq_A y$,
\item or $x,y \in B$ and $x\leq_B y$,
\item or $x\in A$ and $y\in B$.
\end{itemize} 
This product,  defined on isoclasses of posets and also denoted by $\downarrow$, is sometimes called the direct sum \cite{Stanley}. It is associative.
\item $(\tdun{$1$}\tdun{$2$}+\tddeux{$1$}{$2$})\circ(\isoclasse{\mathcal{A}},\isoclasse{\mathcal{B}})$  is the sum of isoclasses of posets $\mathcal{C}$
on $A\sqcup B$ such that:
\begin{itemize}
\item If $x,y\in A$, $x\leq_{\Cal C} y$ if, and only if, $x\leq_A y$.
\item If $x,y\in B$, $x\leq_{\Cal C} y$ if, and only if, $x\leq_B y$.
\item $B < A$ in the sense of the Introduction, that is to say for all $x\in B$, $y\in A$, $x\not\geq_{\Cal C} y$.
\ignore{\item $\Cal C\neq\Cal A\Cal B$ \textcolor{red}{A retirer, ceci est pour $\tddeux{$1$}{$2$}\circ(\isoclasse{\mathcal{A}},\isoclasse{\mathcal{B}})$.
Ajouter $\tdun{$1$}\tdun{$2$}\circ(\isoclasse{\mathcal{A}},\isoclasse{\mathcal{B}})$ donne le terme $\Cal A\Cal B$}
}
\end{itemize}
This product, denoted by $*$,  is associative.
\item  $\tddeux{$1$}{$2$}\blacktriangle(\isoclasse{\mathcal{A}},\isoclasse{\mathcal{B}})=\isoclasse{\mathcal{A}\triangle \mathcal{B}}$, 
where $\mathcal{A}\blacktriangle \mathcal{B}=(A\sqcup B,\leq)$, with, for all $x,y \in A\sqcup B$; $x\leq y$ if and only if:
\begin{itemize}
\item $x,y\in A$ and $x\leq_A y$,
\item or $x,y \in B$ and $x\leq_B y$,
\item or $x\in A$ and $y\in \max(\mathcal{B})$.
\end{itemize} 
This product, defined on isoclasses of posets and also denoted by $\triangle$, is non associative permutative: for all $x,y,z\in F_\mathbb{P}(1)$,
$$x\triangle(y\triangle z)=(xy)\triangle z=y\triangle (x\triangle z).$$
\item  $\tddeux{$1$}{$2$}\blacktriangledown(\isoclasse{\mathcal{A}},\isoclasse{\mathcal{B}})=\isoclasse{\mathcal{A}\triangledown \mathcal{B}}$,
where $\mathcal{A}\blacktriangledown\mathcal{B}=(A\sqcup B,\leq)$, with, for all $x,y \in A\sqcup B$; $x\leq y$ if:
\begin{itemize}
\item $x,y\in A$ and $x\leq_A y$,
\item or $x,y \in B$ and $x\leq_B y$,
\item or $x \in \min(\mathcal{B})$ and $y\in A$.
\end{itemize} 
This product, defined on isoclasses of posets and also denoted by $\triangledown$, is  non associative permutative: 
for all $x,y,z\in F_\mathbb{P}(1)$,
$$x\triangledown(y\triangledown z)=(xy)\triangledown z=y\triangledown (x\triangledown z).$$
\end{enumerate}\end{thm}

\begin{proof} 1. We prove it for $\circ$; the proof is similar in the other cases. First:
\begin{align*}
\tdun{$1$}\tdun{$2$}^{(12)}&=\tdun{$1$}\tdun{$2$};\\
\tdun{$1$}\tdun{$2$}\circ_1 \tdun{$1$}\tdun{$2$}&=\tdun{$1$}\tdun{$2$}\circ_2 \tdun{$1$}\tdun{$2$}=
\tdun{$1$}\tdun{$2$}\tdun{$3$},
\end{align*}
so the product induced on $F_\mathbb{P}(1)$ by $\tdun{$1$}\tdun{$2$}$ is associative and commutative. Moreover:
$$\tdun{$1$}\tdun{$2$}\circ (\isoclasse{\mathcal{A}},\isoclasse{\mathcal{B}})=\isoclasse{(\tdun{$1$}\tdun{$2$}\circ_1 \mathcal{A})\circ_2 \mathcal{B}}.$$
By definition of $\circ$, $(\tdun{$1$}\tdun{$2$}\circ_1 \mathcal{A})=\mathcal{A}\tdun{$2$}$, 
and  $\mathcal{A}\tdun{$2$}\circ_2 \mathcal{B}=\mathcal{A}\mathcal{B}$.\\

2. First:
$$\tddeux{$1$}{$2$}\bullet_1 \tddeux{$1$}{$2$}=\tddeux{$1$}{$2$}\bullet_2 \tddeux{$1$}{$2$}=\tdtroisdeux{$1$}{$2$}{$3$},$$
so the product $\downarrow$ is associative. Moreover:
$$\tddeux{$1$}{$2$}\bullet (\isoclasse{\mathcal{A}},\isoclasse{\mathcal{B}})=\isoclasse{(\tddeux{$1$}{$2$}\bullet_1\mathcal{A})\bullet_2 \mathcal{B}}.$$
The underlying set of the poset $\mathcal{C}=\tddeux{$1$}{$2$}\bullet_1\mathcal{A}$ is $A\sqcup\{2\}$, and:
$$\{(x,y)\in (A\sqcup\{2\})^2\mid x\leq_C y\}=\{(x,y)\in A\mid x\leq_A y\}\sqcup ((A\sqcup\{2\})\times \{2\}).$$
Hence, the underlying set of the poset $S=(\tddeux{$1$}{$2$}\bullet_1\mathcal{A})\bullet_2 \mathcal{B}=$ is $A\sqcup B$, and:
$$\{(x,y)\in (A\sqcup B)^2\mid x\leq_S y\}=\{(x,y)\in A\mid x\leq_A y\}\sqcup \{(x,y)\in B\mid x\leq_B y\}
\sqcup (A\times B).$$

3. First:
\begin{align*}
(\tdun{$1$}\tdun{$2$}+\tddeux{$1$}{$2$})\circ_1 (\tdun{$1$}\tdun{$2$}+\tddeux{$1$}{$2$})
&=\tdun{$1$}\tdun{$2$}\tdun{$3$}+\tdun{$1$}\tddeux{$2$}{$3$}+\tdun{$2$}\tddeux{$1$}{$3$}+\tdun{$3$}\tddeux{$1$}{$2$}
+\tdtroisun{$1$}{$3$}{$2$}+\pdtroisun{$3$}{$1$}{$2$}+\tdtroisdeux{$1$}{$2$}{$3$}\\
&=(\tdun{$1$}\tdun{$2$}+\tddeux{$1$}{$2$})\circ_2 (\tdun{$1$}\tdun{$2$}+\tddeux{$1$}{$2$}),
\end{align*}
so $*$ is associative. Let us compute $\tddeux{$1$}{$2$}\circ(\mathcal{A},\mathcal{B})=(\tddeux{$1$}{$2$}\circ_1 \mathcal{A})\circ_2 \mathcal{B}$. 
By definition, $\tddeux{$1$}{$2$}\circ_1 \mathcal{A}$ is the sum of the posets $\mathcal{C}$ on $A\sqcup \{2\}$ such that:
\begin{itemize}
\item For all $x,y\in A$, $x \leq_\mathcal{C} y$ if, and only if, $x\leq_A y$.
\item For all $y \in A$, we do not have $2\leq_\mathcal{C} y$.
\item There exists $x\in A$, $x\leq_\mathcal{C} 2$.
\end{itemize}
Hence, $\tddeux{$1$}{$2$}\circ(\mathcal{A},\mathcal{B})$ is the sum of all the posets $\mathcal{S}$ on $A\sqcup B$ such that:
\begin{itemize}
\item For all $x,y\in A$, $x \leq_\mathcal{S} y$ if, and only if, $x\leq_A y$.
\item For all $x,y\in B$, $x \leq_\mathcal{S} y$ if, and only if, $x\leq_B y$.
\item For all $x\in B$, $y\in A$, we do not have $x\leq_\mathcal{S} y$.
\item There exists $x\in A$, $y\in B$, such that $x \leq_\mathcal{S} y$.
\end{itemize}
These conditions are equivalent to:
\begin{itemize}
\item For all $x,y\in A$, $x \leq_\mathcal{S} y$ if, and only if, $x\leq_A y$.
\item For all $x,y\in B$, $x \leq_\mathcal{S} y$ if, and only if, $x\leq_B y$.
\item $A<B$.
\item $\mathcal{S} \neq \mathcal{A}\mathcal{B}$.
\end{itemize}
Finally, $\tddeux{$1$}{$2$}\circ(\isoclasse{\mathcal{A}},\isoclasse{\mathcal{B}})$ is the sum of isoclasses of posets $\Cal S$
on $A\sqcup B$ such that:
\begin{itemize}
\item If $x,y\in A$, $x\leq_\mathcal{S} y$ if, and only if, $x\leq_A y$.
\item If $x,y\in B$, $x\leq_\mathcal{S} y$ if, and only if, $x\leq_B y$.
\item $A<_{\Cal S}B$.
\item $\mathcal{S}\neq \mathcal{A}\mathcal{B}$.
\end{itemize}
The first point gives the result for $*$.\\

4. First: 
$$\tddeux{$1$}{$2$}\blacktriangle_2 \tddeux{$1$}{$2$}=\tddeux{$1$}{$2$}\blacktriangle_1 \tdun{$1$}\tdun{$2$}
=\pdtroisun{$3$}{$1$}{$2$}.$$
So for all $x,y,z\in F_\mathbb{P}(1)$, $x\triangle(y\triangle z)=(xy)\triangle z$. The commutativity of the product $m$ gives the permutativity of $\triangle$.
Moreover, $\isoclasse{\mathcal{A}}\triangle \isoclasse{\mathcal{B}}=\isoclasse{(\tddeux{$1$}{$2$}\blacktriangle_1 
\mathcal{A})\blacktriangle_2 \mathcal{B}}$; the underlying set of the poset
$\mathcal{C}=\tddeux{$1$}{$2$}\blacktriangle_1 \mathcal{A}$ is $A\sqcup\{2\}$, and:
$$\{(x,y)\in (A\sqcup\{2\})^2\mid x\leq_C y\}=\{(x,y)\in A^2\mid x\leq_A y\}\sqcup((A\sqcup\{2\}) \times \{2\}).$$
so the underlying set of  $S=(\tddeux{$1$}{$2$}\blacktriangle_1 \mathcal{A})\blacktriangle_2 \mathcal{B}$ is $A\sqcup B$ and:
$$\{(x,y)\in (A\sqcup B)^2\mid x\leq_S y\}=\{(x,y)\in A\mid x\leq_A y\}\sqcup \{(x,y)\in B\mid x\leq_B y\}\sqcup (A \times \max(\mathcal{B})).$$

5. Comes from the preceding point, using the involution on posets. \end{proof}

The products $m$, $\downarrow$ and $*$ are extended to $\overline{F}_\mathbb{P}(1)=\mathbb{K}\oplus F_\mathbb{P}(1)$,
by assuming that $1\in \mathbb{K}$ is the unit for all these products. We now identify $1$ with the empty poset.

\begin{rmk} 
We could also work with the free $\mathbb{P}$-algebra $F_\mathbb{P}(D)$ generated by a set $D$:
this is the vector space generated by isoclasses of posets decorated by $D$ , that is to say pairs $(\mathcal{A},d)$, 
where $\mathcal{A}=(A,\leq_A)$ is a poset and $d:A\longrightarrow D$ is a map.
\end{rmk}

\subsection{Coproducts}

We identify $\overline{F}_\mathbb{P}(1)$ and its graded dual, via the pairing defined on two isoclasses of posets $\isoclasse{\mathcal{A}}$, 
$\isoclasse{\mathcal{B}}$ by:
$$\langle \isoclasse{\mathcal{A}},\isoclasse{\mathcal{B}}\rangle=s_{\isoclasse{\mathcal{A}}} \delta_{\isoclasse{\mathcal{A}},\isoclasse{\mathcal{B}}},$$
where $s_{\isoclasse{\mathcal{A}}}$ is the number of poset automorphisms of $\mathcal{A}$.
We now define two coproducts on $F_\mathbb{P}(1)$. Let $\mathcal{A} \in \mathbb{P}_A$.
\begin{itemize}
\item  We decompose it as $\mathcal{A}=\mathcal{A}_1\ldots \mathcal{A}_k$, where $\mathcal{A}_i$ are the connected components of $\mathcal{A}$. Then:
$$\Delta(\isoclasse{\mathcal{A}})=\sum_{I\subseteq \{1,\ldots,k\}} \isoclasse{\prod_{i\in I} \mathcal{A}_i}\otimes \isoclasse{\prod_{i\notin I} \mathcal{A}_i}.$$
This coproduct is dual of the product $m$.
\item We put:
$$\Delta_*(\isoclasse{\mathcal{A}})=\sum_{I\subseteq A,\: A\setminus I < I} \isoclasse{\mathcal{A}_{\mid A\setminus I}}
\otimes \isoclasse{A_{\mid I}}.$$
This coproduct is the one of the Introduction, and it is the dual of the product $*$.
\end{itemize}

\begin{thm}\begin{enumerate}
\item $(\overline{F}_\mathbb{P}(1),*,\Delta)$ and $(\overline{F}_\mathbb{P}(1),m,\Delta_*)$ are dual  bialgebras.
\item  $(\overline{F}_\mathbb{P}(1),\downarrow,\Delta_*)$ is an infinitesimal bialgebra \cite{Loday1}.
\end{enumerate}\end{thm}

\begin{proof} Let $\mathcal{A}=(A,\leq_A)$ be a finite poset, and $I\subseteq A$. We shall say that $I$ is an ideal of $\mathcal{A}$ if
$A\setminus I<I$.\\

1. For $(\overline{F}_\mathbb{P}(1),m,\Delta_*)$, it remains only to prove the compatibility of $m$ and $\Delta_*$.
Let $\mathcal{A}$ and $\mathcal{B}$ be finite posets. The ideals of $\mathcal{A}\mathcal{B}$ are the subposets $I=I_1 I_2$, where $I_j$ is an ideal of 
$\mathcal{A}_j$ for all $j$. So:
\begin{align*}
\Delta_*(\isoclasse{\mathcal{A}\mathcal{B}})&=\sum_{\mbox{\scriptsize $I$ ideal of $\mathcal{A}$}}
\sum_{\mbox{\scriptsize $J$ ideal of $\mathcal{B}$}}
\isoclasse{\mathcal{A}\mathcal{B}_{\mid A\sqcup B \setminus I\sqcup J}}\otimes \isoclasse{\mathcal{A}\mathcal{B}_{\mid I \sqcup J}}\\
&=\sum_{\mbox{\scriptsize $I$ ideal of $\mathcal{A}$}}\sum_{\mbox{\scriptsize $J$ ideal of $\mathcal{B}$}}
\isoclasse{\mathcal{A}_{\mid A\setminus I}}\isoclasse{\mathcal{B}_{\mid B\setminus J}}\otimes \isoclasse{\mathcal{A}_{\mid I}}
\isoclasse{\mathcal{B}_{\mid J}}\\
&=\Delta_*(\mathcal{A})\Delta_*(\mathcal{B}).
\end{align*}
Hence, $(\overline{F}_\mathbb{P}(1),m,\Delta_*)$ is a bialgebra. By duality, $(\overline{F}_\mathbb{P}(1),*,\Delta)$ also is.\\

2. It remains to prove the compatibility between $\downarrow$ and $\Delta_*$. Let $\mathcal{A}$ and $\mathcal{B}$ be two finite posets.
The ideals of $\mathcal{A}\downarrow \mathcal{B}$ are the ideals of $\mathcal{B}$ and the ideals $I\sqcup B$, where $I$ is an ideal of $\mathcal{A}$.
Note that in this description, $B$ appears two times, as $B$ and $\emptyset \sqcup B$. Hence:
\begin{align*}
\Delta_*(\isoclasse{\mathcal{A}\downarrow \mathcal{B}})&=\sum_{\mbox{\scriptsize $I$ ideal of $\mathcal{A}$}} 
\isoclasse{\mathcal{A}\downarrow\mathcal{B}_{\mid A\setminus I}}\otimes \isoclasse{\mathcal{A}\downarrow \mathcal{B}_{\mid I\sqcup B}}
+\sum_{\mbox{\scriptsize $I$ ideal of $\mathcal{B}$}} 
\isoclasse{\mathcal{A}\downarrow\mathcal{B}_{\mid A\sqcup B\setminus I}}\otimes \isoclasse{\mathcal{A}\downarrow \mathcal{B}_{\mid I}}\\
&-\isoclasse{\mathcal{A}\downarrow\mathcal{B}_{\mid A}}\otimes \isoclasse{\mathcal{A}\downarrow\mathcal{B}_{\mid B}}\\
&=\sum_{\mbox{\scriptsize $I$ ideal of $\mathcal{A}$}}
 \isoclasse{\mathcal{A}_{\mid A\setminus I}} \otimes\isoclasse{\mathcal{A}_{\mid I}}
\downarrow \isoclasse{\mathcal{B}}
+\sum_{\mbox{\scriptsize $I$ ideal of $\mathcal{B}$}} \isoclasse{\mathcal{A}}\downarrow \isoclasse{\mathcal{B}_{\mid B\setminus I}} 
\otimes \isoclasse{\mathcal{B}_{\mid I}}-\isoclasse{\mathcal{A}}\otimes \isoclasse{\mathcal{B}}\\
&=\Delta_*(\isoclasse{\mathcal{A}})\downarrow(1\otimes \isoclasse{\mathcal{B}})+(\isoclasse{\mathcal{A}}\otimes 1)
\downarrow \Delta_*(\isoclasse{\mathcal{B}})-\isoclasse{\mathcal{A}}\otimes \isoclasse{\mathcal{B}}.
\end{align*}
So $(\overline{F}_\mathbb{P}(1),\downarrow,\Delta_*)$ is indeed an infinitesimal bialgebra. \end{proof}

All these objects are graded by the cardinality of posets and are connected. By the rigidity theorem for connected infinitesimal bialgebras \cite{Loday1},
$(\overline{F}_\mathbb{P}(1),\downarrow,\Delta_*)$ is isomorphic to a tensor algebra, with the concatenation product and the deconcatenation coproduct.
Consequently:

\begin{cor}
The commutative bialgebra $(\overline{F}_\mathbb{P}(1),m,\Delta_*)$ is cofree;
the cocommutative bialgebra $(\overline{F}_\mathbb{P}(1),*,\Delta)$ is free.
\end{cor}

\section{Suboperads generated in degree $2$}\label{sect:suboperads}

\subsection{WN Posets}

Let $\mathcal{A}$ be a finite poset. We shall say that it is a \textsl{poset without N}, or \textsl{WN poset}, if it does not contain any subposet isomorphic to $\pquatresix$ \cite{Stanley,Fposets}.
For any finite set $A$, the space of finite WN poset structures on $A$ is denoted by $\mathbb{WNP}_A$.
We define in this way a linear species $\mathbb{WNP}$. For example, here are the isoclasses of \ignore{$\triangledown$-compatible} WN posets of cardinality $\leq 4$:
$$\tun;\hspace{.5cm}\tun\tun,\tdeux;\hspace{.5cm}\tun\tun\tun,\tun\tdeux,\ttroisun,\ttroisdeux,\ptroisun;$$
$$\tun\tun\tun\tun,\tun\tun\tdeux,\tdeux\tdeux,\ttroisun\tun,\ttroisdeux\tun,\ptroisun\tun,
\tquatreun,\tquatredeux,\tquatrequatre,\tquatrecinq,\pquatreun,\pquatredeux,\pquatrequatre,\pquatresept,\pquatrehuit.$$

\begin{lem} \begin{enumerate}
\item Let $\mathcal{A}$ be a WN poset. It can be written in a unique way as $\mathcal{A}=\mathcal{A}_1\downarrow \ldots \downarrow \mathcal{A}_k$, 
with for all $1\leq i \leq k$, $\isoclasse{\mathcal{A}_i}=\tun$ or $\mathcal{A}_i$ is not connected.
\item Let $\mathcal{A}$ and $\mathcal{B}$ be two finite posets. The following conditions are equivalent:
\begin{enumerate}
\item $\mathcal{A}$ and $\mathcal{B}$ are WN.
\item $\mathcal{A}\mathcal{B}$ is WN.
\item $\mathcal{A}\downarrow\mathcal{B}$ is WN.
\end{enumerate}
\end{enumerate}\end{lem}

\begin{proof} (1) {\it Existence}. We proceed by induction on $n=|\mathcal{A}|$. If $n=1$, then $\isoclasse{\mathcal{A}}=\tun$ and the result is obvious. 
Let us assume the result at all rank $<n$. If $\mathcal{A}$ is not connected, we choose $k=1$ and $\mathcal{A}_1=\mathcal{A}$. 
Let us assume that $\mathcal{A}$ is connected. If $\mathcal{A}$ has a unique maximal element $M$, then for all $x\in \mathcal{A}$, 
$x \leq_A M$ so $\mathcal{A}=(\mathcal{A}\setminus \{M\})\downarrow \{M\}$:
we then apply the induction hypothesis to $\mathcal{A}\setminus \{M\}$ and obtain a decomposition of $\mathcal{A}$. 
Let us assume that $\mathcal{A}$ has at least two maximal elements.  We put:
\begin{align*}
\mathcal{A}'&=\{x\in \mathcal{A}\mid \forall M\in \max(\mathcal{A}'), x\leq_A M\},& \mathcal{A}''=\mathcal{A}\setminus \mathcal{A}'.
\end{align*}
Let $M \in \max(\mathcal{A})$ and $M'$ be another maximal element of $\mathcal{A}$. 
We do not have $M\leq M'$, so $M \in \mathcal{A}''$, hence $\max(\mathcal{A})\subseteq \mathcal{A}''$.
Let $m\in \min(\mathcal{A})$, and let us assume that $m\notin \mathcal{A}'$. So there exists a maximal element $M'$, such that $m$ and $M'$ 
are not comparable in $\mathcal{A}$. Moreover, there exists $m\in \max(\mathcal{A})$, $m\leq_A M$. As $\mathcal{A}$ is connected, 
there exists elements $x_1,\ldots,x_{2k-1}$ such that 
$x_0=M\geq_\mathcal{A} x_1\leq_A x_2\geq_\mathcal{A}\ldots \geq_\mathcal{A} x_{2k-1} \leq_A x_{2k}=M'$. 
We choose these elements in such a way that $k$ is minimal.
As $\mathcal{A}$ is WN, necessarly $k=1$, and then $M,m,x_1,M'$ is a copy of $\pquatresix$ in $\mathcal{A}$: contradiction. 
So $\min(\mathcal{A})\subseteq \mathcal{A}'$: consequently, $\mathcal{A}'$ and $\mathcal{A}''$ are nonempty.\\

Let $x \in \mathcal{A}'$ and $y\in \mathcal{A}''$. If $y\in \max(\mathcal{A})$, by definition of $\mathcal{A}'$, $x\leq_A y$. If not, as $y\in \mathcal{A}''$, 
there exists $M\in \max(\mathcal{A})$ such that $y$ and $M$ are not comparable in $\mathcal{A}$; there exists $M'\in \max(\mathcal{A})$, such that $y\leq_A M'$. As $x \in \mathcal{A}'$, $x\leq_A M,M'$.
As $\{x,y,M,M'\}$ cannot be isomorphic to $\pquatresix$, necessarily $x\leq_A y$. Finally, $\mathcal{A}=\mathcal{A}'\downarrow \mathcal{A}''$.
Applying the induction hypothesis to $\mathcal{A}'$ and $\mathcal{A}''$, we obtain the result.\\

{\it Unicity.} Let us assume that $\mathcal{A}=\mathcal{A}_1\downarrow \ldots \downarrow \mathcal{A}_k
=\mathcal{A}'_1\downarrow\ldots\downarrow \mathcal{A}'_l$, with for all $i,j$, $\mathcal{A}_i$ and $\mathcal{A}'_j$ reduced to a single point
or non connected. We proceed by induction on $k$. First, observe that 
if $k \geq 2$, $\mathcal{A}$ is connected and not reduced to a single point. So $k=1$ is equivalent to $l=1$, which proves the result for $k=1$.
Let us assume the result at rank $k-1$. We consider $\mathcal{A}'=\{x\in \mathcal{A}\mid \forall M\in \max(\mathcal{A}), x<_\mathcal{A} M\}$.
Clearly, $\mathcal{A}_1\downarrow\ldots \downarrow \mathcal{A}_{k-1}\subseteq \mathcal{A}'$. If $\mathcal{A}_k=\{M\}$, 
then $M$ is the unique maximal element of $\mathcal{A}$, and $M\notin \mathcal{A}'$. If $\mathcal{A}_k$ is not reduced to a single element, 
then it is not connected. Let $x\in \mathcal{A}_k$
and let $M$ be a maximal element of $\mathcal{A}_k$ which is in a different connected component. Then $M$ is a maximal element of $\mathcal{A}$
and we do not have $x<_\mathcal{A} M$, so $x\notin \mathcal{A}'$. Consequently, $\mathcal{A}'=\mathcal{A}_1\downarrow \ldots \downarrow  \mathcal{A}_{k-1}$. 
Similarly, $\mathcal{A}'=\mathcal{A}_1\downarrow \ldots \downarrow \mathcal{A}'_{l-1}$. Taking the complement, $\mathcal{A}_k=\mathcal{A}'_l$. 
Using the induction hypothesis, $k=l$, and $\mathcal{A}_i=\mathcal{A}'_i$ for all $1\leq i\leq k-1$. 

(2)  (a) $\Longleftrightarrow$ (b). Let $\mathcal{N}$ be a copy of $\pquatresix$ in $\mathcal{A}\mathcal{B}$. As $\pquatresix$ is connected,
it is included in a component of $\mathcal{AB}$, so is included in $A$ or is included in $B$. So $\mathcal{AB}$ is WN
if, and only if, $\mathcal{A}$ and $\mathcal{B}$ are.\\

(a) $\Longleftrightarrow$ (c).
Let $\mathcal{N}$ be a copy of $\pquatresix$ in $\mathcal{A}\downarrow\mathcal{B}$. We put $N_1=\mathcal{N}\cap A$ and $N_2=\mathcal{N}\cap B$.
Then, for all $x\in N_1$, $y\in N_2$, $x\leq_{\mathcal{A}\downarrow\mathcal{B}} y$, 
so $\mathcal{N}=\mathcal{N}_{\mid N_1}\downarrow \mathcal{N}_{\mid N_2}$.
The only possibilities are $(N_1,N_2)=(\mathcal{N},\emptyset)$ or $(\emptyset,\mathcal{N})$, so $\mathcal{N}\subseteq A$ or $\mathcal{N}\subseteq B$.
Hence, $\mathcal{A}\downarrow \mathcal{B}$ is WN if, and only if, $\mathcal{A}$ and $\mathcal{B}$ are. \end{proof}

The vector space $F_\mathbb{WNP}(1)$ generated by the set of isoclasses of WN posets is an algebra for both products
$m$ and $\downarrow$: this is a $2$-As algebra, with the terminology of \cite{Loday2}. As the first product is commutative,
we shall say that it is a Com-As algebra.

\begin{thm}
$\mathbb{WNP}$ is a suboperad of $(\mathbb{P},\bullet)$. It is generated by $m=\tdun{$1$}\tdun{$2$}$ and 
$\downarrow= \tddeux{$1$}{$2$} \in \mathbb{P}_{\{1,2\}}$,
and the relations:
\begin{align*}
m^{(12)}&=m,&m\bullet_1 m&=m\bullet_2 m,&\downarrow\bullet_1\downarrow&=\downarrow\bullet_2\downarrow.
\end{align*}\end{thm}

\begin{proof} {\it First step.} 
 Let us first prove that $\mathbb{WNP}$ is a suboperad of $\mathbb{P}$ for the product $\bullet$. Let $\mathcal{A},\mathcal{B}$ be two WN posets, 
 and $a\in \mathcal{A}$. Let us assume that $\mathcal{C}=\mathcal{A}\bullet_a \mathcal{B}$ contains a copy $N$ of $\pquatresix$. 
 The elements of $N$ are denoted in this way:
$$
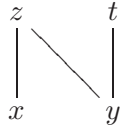
\xymatrix{z\ar@{-}[d]\ar@{-}[rd]&t\ar@{-}[d]\\x&y}$$
As $\mathcal{A}$ and $\mathcal{B}$ are WN, $N$ cannot be included in 
$\mathcal{A}\setminus \{a\}$, nor in $\mathcal{B}$. If $N\cap \mathcal{B}$ is a singleton, $\mathcal{C}/\mathcal{B}\to\hskip -3mm\to\{a\}=\mathcal{A}$ 
contains a copy of $\pquatresix$, formed by $a$ and the three elements of $N\cap \mathcal{A}\setminus\{a\}$: this is a contradiction. 
If $N\cap \mathcal{B}$ contains two elements, six cases are possible:
\begin{itemize}
\item $N\cap \mathcal{B}=\{x,y\}$: as $y\leq_C t$, we should have $x\leq_C t$: this is a contradiction.
\item $N\cap \mathcal{B}=\{x,z\}$: as $y\leq_C z$, we should have $x\leq_C z$: this is a contradiction.
\item $N\cap \mathcal{B}=\{x,t\}$: as $x\leq_C z$, we should have $t\leq_C z$: this is a contradiction.
\item $N\cap \mathcal{B}=\{y,z\}$: as $x\leq_C z$, we should have $x\leq_C y$: this is a contradiction.
\item $N\cap \mathcal{B}=\{y,t\}$: as $y\leq_C z$, we should have $t\leq_C z$: this is a contradiction.
\item $N\cap \mathcal{B}=\{z,t\}$: as $x\leq_C z$, we should have $x\leq_C t$: this is a contradiction.
\end{itemize}
If $N\cap \mathcal{B}$ contains three elements, four cases are possible:
\begin{itemize}
\item $N\cap \mathcal{B}=\{x,y,z\}$: as $y\leq_C t$, we should have $x\leq_C t$: this is a contradiction.
\item $N\cap \mathcal{B}=\{x,y,t\}$: as $y\leq_C z$, we should have $t\leq_C z$: this is a contradiction.
\item $N\cap \mathcal{B}=\{x,z,t\}$: as $y\leq_C t$, we should have $y\leq_C x$: this is a contradiction.
\item $N\cap \mathcal{B}=\{y,z,t\}$: as $x\leq_C z$, we should have $x\leq_C y$: this is a contradiction.
\end{itemize}
In all cases, we obtain a contradiction, so $\mathcal{C}$ is WN. \\

{\it Second step.} We denote by $\mathbb{P}'$ the suboperad of $\mathbb{P}$ generated by $\tdun{$1$}\tdun{$2$}$ and $\tddeux{$1$}{$2$}$.
By the first point, $\mathbb{P}'\subseteq \mathbb{WNP}$. Let us prove the inverse inclusion. Let $\mathcal{A}$ be a WN poset, let us prove
that $\mathcal{A}\in \mathbb{P}'$ by induction on $n=|\mathcal{A}|$. This is obvious if $n=1$ or $n=2$. 
Let us assume the result at all rank $<n$, with $n \geq 3$. If $\mathcal{A}$ is not connected, we can write $\mathcal{A}=\mathcal{A}_1\mathcal{A}_2$, 
with $\mathcal{A}_1$ and $\mathcal{A}_2$ nonempty. By restriction, $\mathcal{A}_1$ and $\mathcal{A}_2$ are WN,
so belong to $\mathbb{P}'$ by the induction hypothesis. Then:
$$\mathcal{A}=\tdun{$1$}\tdun{$2$}\bullet(\mathcal{A}_1,\mathcal{A}_2) \in \mathbb{P}'.$$
If $\mathcal{A}$ is connected, as it is WN we can write it as $\mathcal{A}=\mathcal{A}_1 \downarrow \mathcal{A}_2$,  
with $\mathcal{A}_1$ and $\mathcal{A}_2$ nonempty. By restriction, $\mathcal{A}_1$ and $\mathcal{A}_2$ are WN, so belong 
to $\mathbb{P}'$ by the induction hypothesis. Then:
$$\mathcal{A}=\tddeux{$1$}{$2$}\bullet(\mathcal{A}_1,\mathcal{A}_2) \in \mathbb{P}'.$$

{\it Last step.} In order to give the presentation of $\mathbb{WNP}$ by generators and relations, it is enough to prove that
free $\mathbb{WNP}$-algebras satisfy the required universal property.
We restrict ourselves to $F_\mathbb{WNP}(1)$, the other cases are proved similarly.
More precisely, let $(A,m,\downarrow)$ be a Com-As algebra, that is to say that $(A,m)$ is an associative, commutative algebra, and $(A,\downarrow)$
is an associative algebra, and let $a\in A$. Let us show that there exists a unique morphism $\phi$ of Com-As algebras from $F_\mathbb{WNP}(1)$ to $A$,
sending $\tun$ to $a$.
For this, we consider an iso-class $\isoclasse{\mathcal{A}}$ of a WN poset. 
We define $\phi(\isoclasse{\mathcal{A}})$ by induction on the cardinality of 
$\mathcal{A}$ in the following way:
\begin{itemize}
\item $\phi(\tun)=a$.
\item If $\phi$ is not connected, we put $\mathcal{A}=\mathcal{A}_1\ldots \mathcal{A}_k$, where $k\geq 2$ and $\mathcal{A}_1,\ldots,\mathcal{A}_k$ 
are connected. We put:
$$\phi(\isoclasse{\mathcal{A}})=\phi(\isoclasse{\mathcal{A}_1})\ldots\phi(\isoclasse{\mathcal{A}_k}).$$
As the product $m$ of $A$ is commutative, this does not depend of the chosen order on the connected components of $\mathcal{A}$, so is well-defined.
\item If $\isoclasse{\mathcal{A}}\neq \tun$ and $\mathcal{A}$ is connected, by the preceding lemma it can be uniquely written as 
$\mathcal{A}=\mathcal{A}_1\downarrow \ldots \downarrow \mathcal{A}_k$, with $k\geq 2$, with for all $i$ $\isoclasse{\mathcal{A}_i}=\tun$
or $\mathcal{A}_i$ not connected. We put:
$$\phi(\isoclasse{\mathcal{A}})=\phi(\isoclasse{\mathcal{A}_1})\downarrow\ldots \downarrow \phi(\isoclasse{\mathcal{A}_k}).$$
\end{itemize} 
It is an easy exercise to prove that $\phi$ is indeed a Com-As algebra morphism. \end{proof}

\subsection{Compositions $\blacktriangle$ and $\blacktriangledown$}

We first define the notion of $\triangledown$-compatible poset, inductively on the cardinality. Let $\mathcal{A}=(A,\leq_A)$ be a finite poset of cardinality $n$.
We shall need the following notation: we first put
$$A'=\{y\in A\mid \forall x\in \min(\mathcal{A}), \: x<_\mathcal{A} y\}.$$
We then define $b\mathcal{A}=\mathcal{A}_{\mid A'}$ and $r\mathcal{A}=\mathcal{A}_{\mid A\setminus A'}$.
Note that the minimal elements of $\mathcal{A}$ do not belong to $A'$, so $r\mathcal{A}$ is not empty. It may happen that $b\mathcal{A}$ is empty.

Let us now define $\triangledown$-compatible posets by induction on the cardinality.  If $n=1$, then $\isoclasse{\mathcal{A}}=\tun$ is $\triangledown$-compatible.
If $n \geq 2$, we shall say that $\mathcal{A}$ is $\triangledown$-compatible if one of the following conditions holds:
\begin{enumerate}
\item $\mathcal{A}$ is not connected and all the connected components of $\mathcal{A}$ are $\triangledown$-compatible.
\item $\mathcal{A}$ is connected and the following conditions hold:
\begin{enumerate}
\item $b\mathcal{A}$ is not empty.
\item $b\mathcal{A}$ and $r\mathcal{A}$ are $\triangledown$-compatible.
\item $\mathcal{A}=b\mathcal{A}\triangledown r\mathcal{A}$.
\end{enumerate}\end{enumerate}
The subspecies of $\mathbb{P}$ of $\triangledown$-compatible posets is denoted by $\mathbb{CP}_\triangledown$.\\

For example, here are the isoclasses of $\triangledown$-compatible posets of cardinality $\leq 4$:
$$\tun;\hspace{.5cm}\tun\tun,\tdeux;\hspace{.5cm}\tun\tun\tun,\tun\tdeux,\ttroisun,\ttroisdeux,\ptroisun;$$
$$\tun\tun\tun\tun,\tun\tun\tdeux,\tdeux\tdeux,\ttroisun\tun,\ttroisdeux\tun,\ptroisun\tun,
\tquatreun,\tquatredeux,\tquatrequatre,\tquatrecinq,\pquatreun,\pquatrequatre,\pquatresix,\pquatresept,\pquatrehuit.$$

\begin{lem} Let $\mathcal{A}=(A,\leq_A)$ and $\mathcal{B}=(B,\leq_B)$ be two finite posets.
\begin{enumerate}
\item $b(\mathcal{A}\triangledown \mathcal{B})=\mathcal{A}b\mathcal{B}$ and $r(\mathcal{A}\triangledown \mathcal{B})=r\mathcal{B}$.
\item Let us assume that $\mathcal{A}$ is $\triangledown$-compatible.
\begin{enumerate}
\item $\mathcal{A}$ is connected and different from $\tun$ if, and only if, $b\mathcal{A}$ is nonempty.
\item If $\mathcal{A}$ is connected, then  $\isoclasse{r\mathcal{A}}=\tun$ or $r\mathcal{A}$ is not connected.
\end{enumerate}\end{enumerate}\end{lem}

\begin{proof} (1) This comes easily from the observation that $\min(\mathcal{A}\triangledown \mathcal{B})=\min \mathcal{B}$. \\

(2) (a) $\Longrightarrow$. By definition of the $\triangledown$-compatibility. $\Longleftarrow$. If $\isoclasse{\mathcal{A}}=\tun$, then obviously
$b\mathcal{A}$ is empty. If $\mathcal{A}$ is not connected, let us take $a\in A$. If $m$ is a minimal element of connected component
of $\mathcal{A}$ which does not contain $a$, we do not have $a>_A m$, so $a\notin A'$: $b\mathcal{A}$ is empty.\\

2. (b) Let us assume that $\isoclasse{r\mathcal{A}} \neq \tun$ and is connected. By (2) (a), $b r\mathcal{A}$ is not empty.
By (1), $b\mathcal{A}=b(b\mathcal{A} \triangledown r\mathcal{A})=b\mathcal{A}br\mathcal{A} \neq b\mathcal{A}$, which is a contradiction. \end{proof}

\begin{lem} \label{lemmetriangle}
Let $\mathcal{A}=(A,\leq_A)$ and $\mathcal{B}=(B,\leq_B)$ be two $\triangledown$-compatible posets. Then $\mathcal{A}\mathcal{B}$ 
and $\mathcal{A}\triangledown \mathcal{B}$ are $\triangledown$-compatible.
\end{lem}

\begin{proof} Let $\mathcal{A}=\mathcal{A}_1\ldots \mathcal{A}_k$ and $\mathcal{B}=\mathcal{B}_1\ldots \mathcal{B}_l$ be the decomposition of 
$\mathcal{A}$ and $\mathcal{B}$ into connected components.
By definition, $\mathcal{A}_1,\ldots,\mathcal{A}_k,\mathcal{B}_1,\ldots,\mathcal{B}_l$ are $\triangledown$-compatible. 
Then $\mathcal{A}\mathcal{B}$ is not connected and its connected components
are $\mathcal{A}_1,\ldots,\mathcal{A}_k,\mathcal{B}_1,\ldots, \mathcal{B}_l$, so $\mathcal{A}\mathcal{B}$ is $\triangledown$-compatible.\\

Let us prove that $\mathcal{C}=\mathcal{A}\triangledown \mathcal{B}$ is $\triangledown$-compatible
First, observe that $r\mathcal{C}=r\mathcal{B}$ and $b\mathcal{C}=\mathcal{A} b\mathcal{B}$. 
\begin{itemize}
\item If $\mathcal{B}$ is not connected or reduced to a single element, $r\mathcal{C}=\mathcal{B}$ and $b\mathcal{C}=\mathcal{A}$,
so are both $\triangledown$-compatible.
\item If not, $r\mathcal{C}=r\mathcal{B}$ and $b\mathcal{C}=\mathcal{A} b\mathcal{B}$ are both $\triangledown$-compatible,
as $\mathcal{B}$ is $\triangledown$-compatible. 
\end{itemize}
Moreover:
$$b\mathcal{C}\triangledown r\mathcal{C}=(\mathcal{A} b\mathcal{B})\triangledown r\mathcal{B}
=\mathcal{A} \triangledown (b\mathcal{B}\triangledown r\mathcal{B})=\mathcal{A} \triangledown \mathcal{B}=\mathcal{C},$$
so $\mathcal{C}$ is $\triangledown$-compatible. \end{proof}

\begin{thm}
The species $\mathbb{C\mathcal{A}}_\triangledown$ is a suboperad of $(\mathbb{P},\blacktriangledown)$. It is generated by the elements
$m=\tdun{$1$}\tdun{$2$}$ and $\triangledown=\tddeux{$1$}{$2$}$ in $\mathbb{CP}^\triangledown(\{1,2\})$, and the relations:
\begin{align*}
m^{(12)}&=m,&m\blacktriangledown_1 m&=m\blacktriangledown_2 m,&
\triangledown\blacktriangledown_1 m&=\triangledown\blacktriangledown_2\triangledown.
\end{align*}\end{thm}

\begin{proof} {\it First step.} Let us prove that $\mathbb{CP}^\triangledown$ is a suboperad. Let $\mathcal{A}=(A,\leq_A)$ and $\mathcal{B}=(B,\leq_B)$ be
two $\triangledown$-compatible posets, and let $a\in A$; let us show that $\mathcal{C}=\mathcal{A}\blacktriangledown_a \mathcal{B}$ 
is $\triangledown$-compatible. We proceed by induction on the cardinality $n$ of $\mathcal{A}$. If $n=1$, then $\mathcal{C}=\mathcal{B}$
and the result is obvious. Let us assume the result at all ranks $<n$. If $\mathcal{A}$ is not connected, 
we can write $\mathcal{A}=\mathcal{A}_1\mathcal{A}_2$, with $\mathcal{A}_1$ and $\mathcal{A}_2$ nonempty $\triangledown$-compatible posets. Then:
\begin{align*}
\mathcal{C}&=(\tdun{$1$}\tdun{$2$}\blacktriangledown (\mathcal{A}_1,\mathcal{A}_2))\blacktriangledown_a \mathcal{B}\\
&=\tdun{$1$}\tdun{$2$} \blacktriangledown(\mathcal{A}_1\blacktriangledown_a \mathcal{B},\mathcal{A}_2)
\mbox{ or }\tdun{$1$}\tdun{$2$} \blacktriangledown(\mathcal{A}_1,\mathcal{A}_2\blacktriangledown_a \mathcal{B})\\
&=(\mathcal{A}_1\blacktriangledown_a \mathcal{B})\mathcal{A}_2 \mbox{ or }\mathcal{A}_1 (\mathcal{A}_2 \blacktriangledown_a \mathcal{B}).
\end{align*}
We conclude with the induction hypothesis applied to $\mathcal{A}_1$ or $\mathcal{A}_2$ and with lemma \ref{lemmetriangle}.
If $\mathcal{A}$ is connected, we can write $\mathcal{A}=\mathcal{A}_1\triangledown \mathcal{A}_2$, with $\mathcal{A}_1=b\mathcal{A}$ 
and $\mathcal{A}_2=r\mathcal{A}$ nonempty $\triangledown$-compatible posets. Then:
\begin{align*}
\mathcal{C}&=(\tddeux{$1$}{$2$}\blacktriangledown (\mathcal{A}_1,\mathcal{A}_2))\blacktriangledown_a \mathcal{B}\\
&=\tddeux{$1$}{$2$} \blacktriangledown(\mathcal{A}_1\blacktriangledown_a \mathcal{B},\mathcal{A}_2)
\mbox{ or }\tddeux{$1$}{$2$} \blacktriangledown(\mathcal{A}_1,\mathcal{A}_2\blacktriangledown_a \mathcal{B})\\
&=(\mathcal{A}_1\blacktriangledown_a \mathcal{B})\triangledown \mathcal{A}_2 \mbox{ or }\mathcal{A}_1 \triangledown (\mathcal{A}_2 
\blacktriangledown_a \mathcal{B}).
\end{align*}
We conclude with the induction hypothesis applied to $\mathcal{A}_1$ or $\mathcal{A}_2$ and with lemma \ref{lemmetriangle}. \\

{\it Second step.} Let $\mathbb{P}'$ be the suboperad of $(\mathbb{P},\blacktriangledown)$ generated by $\tdun{$1$}\tdun{$2$}$
and $\tddeux{$1$}{$2$}$. The first step implies that $\mathbb{P}' \subseteq \mathbb{CP}^\triangledown$. Let us prove the inverse inclusion.
Let $\mathcal{A}$ be a $\triangledown$-compatible poset; let us prove that $\mathcal{A} \in \mathbb{P}'$ by induction on its cardinality $n$.
It is obvious if $n=1$. Let us assume the result at all rank $<n$. If $\mathcal{A}$ is not connected, we can write $\mathcal{A}=\mathcal{A}_1\mathcal{A}_2$,
with $\mathcal{A}_1$ and $\mathcal{A}_2$ nonempty $\triangledown$-compatible posets. Then $\mathcal{A}=\tdun{$1$}\tdun{$2$}
\blacktriangledown(\mathcal{A}_1,\mathcal{A}_2)$ belongs to $\mathbb{P}'$ by the induction hypothesis.
If $\mathcal{A}$ is connected, we can write $\mathcal{A}=\mathcal{A}_1\triangledown \mathcal{A}_2$, with $\mathcal{A}_1=b\mathcal{A}$ 
and $\mathcal{A}_2=r\mathcal{A}$ nonempty $\triangledown$-compatible posets. Then
$\mathcal{A}=\tddeux{$1$}{$2$}\blacktriangledown (\mathcal{A}_1,\mathcal{A}_2)$ belongs to $\mathbb{P}'$ be the induction hypothesis. \\

{\it Last step.} In order to prove the generation of $\mathbb{CP}^\triangledown$ by generators and relations, 
it is enough to prove the required  universal property for free $\mathbb{CP}^\triangledown$-algebras. 
Let us restrict ourselves to the case of one generator; the proof is similar in the other cases.
$F_{\mathbb{CP}^\triangledown}(1)$ is the space generated by the isoclasses of $\triangledown$-compatible posets, 
with the products $m$ and $\triangledown$. Let $A$ be an algebra with an associative and commutative product $m$ and 
a second product $\triangledown$, such that:
$$\forall x,y,z \in A,\: x \triangledown (y\triangledown z)=(xy)\triangledown z.$$
Let $a\in A$; let us prove that there exists a unique morphism $\phi:F_{\mathbb{CP}^\triangledown}(1)\longrightarrow A$
for the two products $m$ and $\triangledown$, such that $\phi(\tun)=a$. 

We define $\phi(\isoclasse{\mathcal{A}})$ for any $\triangledown$-compatible poset $\mathcal{A}$ by induction on its cardinality $n$.
If $n=1$, then $\phi(\isoclasse{\mathcal{A}})=a$. If $n>1$ and $\mathcal{A}$ is not connected, let us put $\mathcal{A}=\mathcal{A}_1\ldots \mathcal{A}_k$
be the decomposition of $\mathcal{A}$ into connected components. We then put:
$$\phi(\isoclasse{\mathcal{A}})=\phi(\isoclasse{\mathcal{A}_1})\ldots \phi(\isoclasse{\mathcal{A}_k}).$$
As the product $m$ of $A$ is associative and commutative, this does not depend on the chosen order on the $\mathcal{A}_i$, so is well-defined.
If $n>1$ and $\mathcal{A}$ is connected, we write $\mathcal{A}=b\mathcal{A}\triangledown r\mathcal{A}$ and put:
$$\phi(\isoclasse{\mathcal{A}})=\phi(\isoclasse{b\mathcal{A}})\triangledown \phi(\isoclasse{r\mathcal{A}}).$$
Let us prove that this is indeed an algebra morphism for both products. It is immediate for $m$. 
Let $\mathcal{A}$ and $\mathcal{B}$ be two $\triangledown$-compatible
posets. We put $\mathcal{C}=\mathcal{A}\triangledown \mathcal{B}$. 
If $\mathcal{A}$ is not connected, then $b\mathcal{C}=\mathcal{A}$ and $r\mathcal{C}=\mathcal{B}$. By definition of $\phi$,
$\phi(\isoclasse{\mathcal{C}})=\phi(\isoclasse{\mathcal{A}})\triangledown \phi(\isoclasse{\mathcal{B}})$. 
If $\mathcal{A}$ is not connected, then $b\mathcal{C}=\mathcal{A}b\mathcal{B}$ and $r\mathcal{C}=r\mathcal{A}$. By definition of $\phi$:
\begin{align*}
\phi(\isoclasse{\mathcal{C}})&=\phi(\isoclasse{\mathcal{A}b\mathcal{B}})\triangledown \phi(\isoclasse{r\mathcal{B}})\\
&=(\phi(\isoclasse{\mathcal{A}}) \phi(\isoclasse{b\mathcal{B}}))\triangledown \phi(\isoclasse{r\mathcal{B}})\\
&=\phi(\isoclasse{\mathcal{A}}) \triangledown (\phi(\isoclasse{b\mathcal{B}})\triangledown \phi(\isoclasse{r\mathcal{B}}))\\
&=\phi(\isoclasse{\mathcal{A}})\triangledown \phi(\isoclasse{\mathcal{B}}).
\end{align*}
So the morphism $\phi$ exists. As $F_{\mathbb{CP}^\triangledown(1)}$ is generated for the products $m$ and $\triangledown$ by $\tun$,
the unicity is immediate. \end{proof}

\subsection{Comparison of WN posets and $\triangledown$-compatible posets}

The preceding observations on WN posets and $\triangledown$-compatible posets gives the following decompositions:
\begin{enumerate}
\item \begin{enumerate}
\item If $\mathcal{A}$ is a non connected WN poset, it can be written as $\mathcal{A}=\mathcal{A}_1\ldots \mathcal{A}_k$,
where $\mathcal{A}_1,\ldots,\mathcal{A}_k$ are connected WN posets. This decomposition is unique, up to the order of the factors.
\item If $\mathcal{A}$ is connected and different from $\tun$, it can be uniquely written as $\mathcal{A}=\mathcal{A}_1 \downarrow \mathcal{A}_2$,
where $\mathcal{A}_1$ is a WN poset, and $\mathcal{A}_2$ is a WN poset, non connected or equal to $\tun$.
\end{enumerate}
\item \begin{enumerate}
\item If $\mathcal{A}$ is a non connected $\triangledown$-compatible poset, it can be written as $\mathcal{A}=\mathcal{A}_1\ldots \mathcal{A}_k$,
where $\mathcal{A}_1,\ldots,\mathcal{A}_k$ are connected $\triangledown$-compatible posets. This decomposition is unique, up to the order of the factors.
\item If $\mathcal{A}$ is connected and different from $\tun$, it can be uniquely written as $\mathcal{A}=\mathcal{A}_1 \triangledown \mathcal{A}_2$,
where $\mathcal{A}_1$ is a $\triangledown$-compatible poset, and $\mathcal{A}_2$ is a $\triangledown$-compatible poset,
non connected or equal to $\tun$.
\end{enumerate}\end{enumerate}

Consequently, the species $\mathbb{WNP}$ and $\mathbb{CP}^\triangledown$ are isomorphic (as species, not as operads). 
An isomorphism is inductively defined by a bijection $\theta$ from $\mathbb{WNP}_A$ to $\mathbb{CP}^\triangledown_A$ by:
\begin{itemize}
\item $\theta(\tdun{$a$})=\tdun{$a$}$.
\item If $\mathcal{A}=\mathcal{A}_1\ldots \mathcal{A}_k$ is not connected,
$\theta(\mathcal{A})=\theta(\mathcal{A}_1)\ldots \theta(\mathcal{A}_k)$.
\item If $\mathcal{A}=\mathcal{A}_1\downarrow \mathcal{A}_2$ is connected, with $\mathcal{A}_2$ non connected or equal to $1$, then 
$\theta(\mathcal{A})=\theta(\mathcal{A}_1)\triangledown\theta(\mathcal{A}_2)$.
\end{itemize}
For example, if $\mathcal{A}$ is a WN poset of cardinality $\leq 4$, then $\theta(\mathcal{A})=\mathcal{A}$, except if $\isoclasse{\mathcal{A}}=
\pquatredeux$. Moreover:
$$\theta\left(\pdquatredeux{$1$}{$2$}{$3$}{$4$}\right)
=\theta(\tddeux{$1$}{$2$}\tdun{$3$}\downarrow \tdun{$4$})
=\tddeux{$1$}{$2$}\tdun{$3$}\triangledown \tdun{$4$}=\pdquatresix{$1$}{$3$}{$2$}{$4$}.$$

As a consequence, the sequences $(dim(\mathbb{WNP}_{\{1,\ldots,n\}}))_{n\geq 1}$ and $(dim(\mathbb{CP}^\triangledown_{\{1,\ldots,n\}}))_{n\geq 1}$
are both equal to sequence A048172 of the OEIS \cite{Sloane}; considering the isoclasses,
the sequences $(dim(F_\mathbb{WNP}(1)_n)_{n\geq 1}$ and $(dim(F_{\mathbb{CP}^\triangledown}(1)_n)_{n\geq 1}$
are both equal to sequence A003430 of the OEIS.


\end{document}